\documentclass{daj}

\usepackage[latin9]{inputenc}
\usepackage{amsmath}
\usepackage{amsthm}
\usepackage{amssymb}
\usepackage{setspace}

\makeatletter
\theoremstyle{plain}
\newtheorem{thm}{Theorem}
  \theoremstyle{plain}
  \newtheorem{cor}[thm]{Corrolary}
  \theoremstyle{plain}
  \newtheorem{lem}[thm]{Lemma}

\usepackage{amsfonts}

\DeclareMathOperator{\tower}{\operatorname{Tow}}
\DeclareMathOperator{\codim}{\operatorname{Codim}}
\dajAUTHORdetails{%
  title = {Popular Progression Differences in Vector Spaces II}, 
  author = {Jacob Fox and Huy Tuan Pham},
  plaintextauthor = {Jacob Fox and Huy Tuan Pham},
}   

\dajEDITORdetails{%
   year={2019},
   number={16},
   received={30 August 2017},   
   published={20 November 2019},  
   doi={10.19086/da.11002},       
}   

\begin{document}
\begin{frontmatter}[classification=text]

\title{Popular Progression Differences in Vector Spaces II} 

\author[jf]{Jacob Fox \thanks{Supported by a Packard Fellowship and by NSF grant DMS-1855635.}}
\author[htp]{Huy Tuan Pham\thanks{Supported by the Stanford Undergraduate Research Institute in Mathematics (SURIM).}}

\begin{abstract}
Green used an arithmetic analogue of Szemer\'edi's celebrated regularity lemma to prove the following strengthening of Roth's theorem in vector spaces. 
For every $\alpha>0$, $\beta<\alpha^3$, and prime number $p$, there is a least positive integer $n_p(\alpha,\beta)$ such that if $n \geq n_p(\alpha,\beta)$, then for  every subset of $\mathbb{F}_p^n$ of density at least $\alpha$ there is a nonzero $d$ for which the density of three-term arithmetic progressions with common difference $d$ is at least $\beta$. We determine for $p \geq 19$ the tower height of $n_p(\alpha,\beta)$ up to an absolute constant factor and an additive term depending only on $p$. In particular, if we want half the random bound (so $\beta=\alpha^3/2$), then the dimension $n$ required is a tower of twos of height $\Theta \left((\log p) \log \log (1/\alpha)\right)$.  It turns out that the tower height in general takes on different forms in several different regions of $\alpha$ and $\beta$, and different arguments are used both in the upper and lower bounds to handle these cases. 
\end{abstract}
\end{frontmatter}

\section{Introduction}
The game Set consists of a deck of cards. Each card has four attributes: color, shape, shading, and number, and there are three possibilities for each attribute, for a total of $3^4=81$ cards. The goal of the game is to find a ``set'', which is a triple of distinct cards in which each attribute is the same or all different on the three cards. How many cards can there be which contains no set? We can naturally view each card as an element of $\mathbb{F}_3^4$, and a set is then a line (or, equivalently, a three-term arithmetic progression) in this vector space. While a seemingly recreational problem, its generalization to higher dimensions is the well-known {\it cap set problem}, and is related to major open problems in combinatorics, number theory, and computer science. It asks: what is the maximum size of a subset of $\mathbb{F}_3^n$ which does not contain a line? 

Recently, there was a breakthrough on the cap set problem by Croot, Lev, and Pach \cite{CLP} using the polynomial method. Building on this breakthrough,
Ellenberg and Gijswijt \cite{EG} proved that any subset of $\mathbb{F}_{3}^{n}$ which  does not contain a line has at most $O(2.756^{n})$
elements. In the other direction, an earlier construction of Edel \cite{Edel} gives a subset of $\mathbb{F}_{3}^{n}$ with $\Omega(2.217^{n})$ elements that contains no line.

The senior author can add a personal note. His son David, who was six at the time, was playing with Set cards, and observed that he could find nine cards for which he could make twelve sets among them. In other words, he found an affine two-dimensional subspace of $\mathbb{F}_3^4$. This naturally led us to study the {\it multidimensional cap set problem}: what is the maximum size, denoted by $r(n,m)$, of a subset of $\mathbb{F}_3^n$ which does not contain an affine $m$-dimensional subspace? One might hope that the polynomial method proof of the cap set result would naturally extend to give a good bound for the multidimensional cap set problem. This fails due to the complexity of the linear system defining an affine subspace of dimension $m$ for $m>1$. Nevertheless, there is a simple averaging argument which establishes a multidimensional generalization using the cap set result and induction on the dimension $m$ of the desired affine subspace, which we next describe.

The arithmetic triangle removal lemma of the first author and Lov\'asz \cite{FL} as discussed in detail later in the introduction implies a supersaturation extension of the cap set result. This says that any subset of $\mathbb{F}_3^n$ of density $\alpha$ has three-term arithmetic progression density at least $\alpha^{C}$, where $C \approx 13.901$ is an explicit constant. This includes counting trivial three-term arithmetic progressions (those with common difference zero). 

Let $A$ be a subset of $\mathbb{F}_3^n$ of size $r(n,m)$ which does not contain an affine $m$-dimensional subspace, so $A$ has density $\alpha:=r(n,m)/N$ with $N:=|\mathbb{F}_3^n|=3^n$. From the supersaturation result mentioned above, the set $A$ has at least $\alpha^CN^2$ three-term arithmetic progressions. Let $d$ be a nonzero element of $\mathbb{F}_3^n$ such that the number of three-term arithmetic progressions in $A$ with common difference $d$ is maximum. 
As there are $\alpha N$ trivial arithmetic progressions in $A$, by averaging, the number of three-term arithmetic progressions with common difference $d$ is at least $\left(\alpha^C N^2-\alpha N\right)/\left(N-1\right)$, which is asymptotically $(1-o(1))\alpha^C N$ if $\alpha \geq N^{-1/13}$. Let $S$ be a subspace of $\mathbb{F}_3^n$ of dimension $n-1$ not containing $d$. Let $A'$ be the subset of $S$ which contains all elements $x$ for which the entire three-term arithmetic progression $x,x+d,x+2d$ is in $A$. We have $|A'| \geq (1-o(1))\alpha^{C}N/3$, so $A'$ has density at least $(1-o(1))\alpha^C$ in $S$. If $A'$ contains a $(m-1)$-dimensional subspace, then adding $0$, $d$, and $2d$ to the elements of this affine subspace, we get a $m$-dimensional subspace contained in $A$, a contradiction. Hence, $|A'| \leq r(n-1,m-1)$. Putting this together and using induction on $m$, we easily obtain $$r(n,m) \leq (1+o(1))N^{1-C^{-m}}.$$ 

In the other direction, for $m \geq 1$, a random set of density $\alpha$ has probability at most $\alpha^{3^m}$ of containing a particular affine $m$-dimensional subspace, and there are less than $N^{m+1}$ such subspaces, implying that $$r(n,m) \geq N^{1-(m+1)3^{-m}}.$$ One can improve on this bound 
a bit by considering more advanced probabilistic methods (see \cite{AlSp}) such as the alteration method,
the Lov\'asz local lemma, or considering the $m$-dimensional cap set
process, but none of these would improve on the constant $3$.

The lower and upper bounds for $r(n,m)$ have the same form $N^{1-\epsilon_m}$ with $\epsilon_m \to 0$ exponentially fast in $m$, but with different exponential constants ($3$ in the lower bound and $C$ in the upper bound). Is $3$, the lower bound given by considering a random set, the right exponential constant? 

Note that the upper bound argument picks the nonzero $d$ for which $A$ has the most three-term arithmetic progressions with common difference $d$, while the bound we use only considers the average, which might be substantially less. Indeed, a result of Green \cite{Green05} shows that we can find a ``popular'' $d$  for which the density of three-term arithmetic progressions with common difference $d$ is arbitrarily close to the random bound of $\alpha^3$, provided that the dimension of the space is sufficiently large, while the global density of three-term arithmetic progressions may be substantially smaller. 
 
\begin{thm}[Green \cite{Green05}]
\label{thm:Original theorem} For each prime $p$, $\alpha>0$, and $\beta<\alpha^3$,
there is a least positive integer $n_{p}(\alpha,\beta)$ such that the
following holds. For each $n\geq n_{p}(\alpha,\beta)$ and every subset
$A$ of $\mathbb{F}_{p}^{n}$ with density at least $\alpha$, there is a nonzero
$d$ in $\mathbb{F}_{p}^{n}$ such that the density of three-term
arithmetic progressions with common difference $d$ in $A$
is at least $\beta$. 
\end{thm}

The condition $\beta < \alpha^3$ is necessary, as a random set shows that $n_p(\alpha,\beta)$ cannot exist for $\beta > \alpha^3$, and a more involved construction also rules out the case $\beta=\alpha^3$.

Thus, in the upper bound argument for $r(n,m)$, if $n \geq n_3(\alpha,\beta)$ with $\beta=\alpha^{3+o(1)}$, we could improve the lower bound on the density of $A'$ from $(1-o(1))\alpha^C$ to $\beta$, which would imply by induction on $m$ that $3$ is the correct exponential constant in $r(n,m)$. The problem with this approach is that $n_3(\alpha,\beta)$ may be much larger than $n$ when $\alpha \le N^{-C^{-m}}$. The proof of Green \cite{Green05} uses an arithmetic regularity lemma (an arithmetic analogue of Szemeredi's celebrated graph regularity lemma \cite{Sz76}) and gives a tower-type upper bound for $n_3(\alpha,\beta)$. This is much larger than the bound that would be needed for the approach above in estimating $r(n,m)$ to work.  In the first part \cite{FPI} of this two-part sequence of papers, we determined that, for $p$ and $\alpha$ fixed and $\beta=\alpha^3-\epsilon$, the function $n_p(\alpha,\beta)$  grows as an exponential tower of $p$'s of height $\Theta(\log(1/\epsilon))$ as $\epsilon\to 0$. This might suggest that the bound on $n_3(\alpha,\beta)$ is likely too large to be useful in determining the exponential constant for $r(n,m)$. Still, we allow for a considerably smaller value of $\beta$, so one might still have hope for this approach. However, the main result in this paper determines the order of the tower for $n_p(\alpha,\beta)$ for $p \geq 19$, and the tower height grows for $\beta=\alpha^{3+o(1)}$, giving stronger evidence that this approach fails to determine the exponential constant in $r(n,m)$.

We next describe the growth of $n_p(\alpha,\beta)$, which has different behavior depending on the choice of $\alpha$ and $\beta$. We first handle the case $\beta>0$ is sufficiently small as a function of $\alpha$ and $p$. By supersaturation, the cap set problem is equivalent to the special case of estimating $n_p(\alpha,\beta)$ when $\beta>0$ is sufficiently small as a function of $\alpha$ and $p$. For such $\beta$, 
$n_p(\alpha,\beta)$ is the least integer such that for all $n \geq n_p(\alpha,\beta)$, every subset of $\mathbb{F}_p^n$ of density at least $\alpha$ contains a nontrivial three-term arithmetic progression.
In particular, for $\beta$ sufficiently small, we can use the recently established polynomial
bound by the first author and Lov\'asz \cite{FL} on the arithmetic
removal lemma to show that $n_{p}(\alpha,\beta)$ is logarithmic in
$1/\alpha$. Precisely, for $0<\beta\leq\alpha^{C_{p}}/2$, where
$C_{p}=\Theta(\log p)$ is an explicit constant, we have $n_{p}(\alpha,\beta)=\Theta(\log(1/\alpha))$.
The lower bound follows by considering the largest possible subset
of $\mathbb{F}_{p}^{n}$ without a nontrivial three-term arithmetic
progression; Alon, Shpilka, and Umans \cite{ASU} observed that a variant of Behrend's
construction shows that such a subset has size at least $\left((p+1)/2\right)^{n-o(n)}$
for $p \geq 3$ fixed. To show the upper bound, assume $n\geq C_{p}\log_p(2/\alpha) = \Theta(\log (1/\alpha))$.
By the arithmetic removal lemma, as discussed in the next
paragraph, any subset of $\mathbb{F}_{p}^{n}$ of density $\alpha$
has three-term arithmetic progression density (this includes
those with common difference zero) at least $\alpha^{C_{p}}$. As
the density of three-term arithmetic progressions with common difference
zero is the density $\alpha$ of the set, by averaging, there is a
nonzero $d$ for which the density of three-term arithmetic progressions
with common difference $d$ is at least $\frac{\alpha^{C_{p}}p^{n}-\alpha}{p^{n}-1}\geq\alpha^{C_{p}}/2$.

A version of Green's arithmetic removal lemma in $\mathbb{F}_{p}^{n}$
from \cite{FL} states that for each $\epsilon>0$ there is $\delta=\delta(\epsilon)$
such that if $X=\{x_{i}\}_{i=1}^{m}$, $Y=\{y_{i}\}_{i=1}^{m}$, $Z=\{z_{i}\}_{i=1}^{m}$
are subsets of $\mathbb{F}_{p}^{n}$ with $m\geq\epsilon p^{n}$ and
$x_{i},y_{i},z_{i}$ form a three-term arithmetic progression for
each $i$, then there are at least $\delta p^{2n}$ three-term arithmetic
progressions $x_{i},y_{j},z_{k}$. Green's proof uses the arithmetic
regularity lemma and gives an upper bound on $1/\delta$ which is
a tower of twos of height $\epsilon^{-O(1)}$. Recently, it was observed
by Blasiak et al.~\cite{BCCGNSU} and Alon that the recent breakthrough
on the cap set problem extends to prove a multicolor sum-free result,
and results of Kleinberg, Sawin, and Speyer \cite{KSS}, Norin \cite{Norin},
and Pebody \cite{Pebody} show that the bound for the multicolor sum-free
result is sharp. Using this result, the first author and Lov\'asz \cite{FL} proved $\delta(\epsilon)\geq\epsilon^{C_{p}}$, which is essentially tight.
Note that taking $X=Y=Z=A$, we have that any subset $A\subset\mathbb{F}_{p}^{n}$
of density $\alpha$ has three-term arithmetic progression density
at least $\alpha^{C_{p}}$.

It is an interesting problem to understand how $n_{p}(\alpha,\beta)$
grows as we increase $\beta$. This function is on the order of $\log(1/\alpha)$
when $0<\beta\leq\alpha^{C_{p}}/2$, and is a tower of $p$'s of height
$\Theta(\log(1/\epsilon))$ for $\alpha$ fixed and $\epsilon$ small, where $\epsilon = \alpha^3-\beta$. 

We determine for $p \geq 19$ the tower height of $n_p(\alpha,\beta)$ up to an absolute constant factor and an additive constant depending on $p$. One of the difficulties in doing this is that the tower height takes a different form in different regions of $\alpha$ and $\beta$, and the proofs use additional ideas beyond those in the proof of the main result in \cite{FPI} both in the upper and lower bounds.

We first discuss the case when $\alpha \leq 1/2$. When $0<\beta<\alpha^{3+e^{-133}}$, the discussion above and Theorem \ref{thm:largebound} together show that $n_p(\alpha,\beta)$ grows as an exponential tower of constant height (depending on $p$) with $\log(1/\alpha)$ on top. So we  assume $\beta>\alpha^{3+e^{-133}}$. This case splits into three cases, depending on whether or not $\epsilon$ is small, in an intermediate range, or large, with respect to $\alpha$ and $p$, where $\epsilon=\alpha^3-\beta$. In particular, when $\epsilon<\alpha^3/(\log (1/\alpha))^{\log p}$, $n_p(\alpha,\beta)$ grows as a tower of $p$'s of height $\Theta(\log (\alpha^3/\epsilon))\pm O_p(1)$.\footnote{Here, and throughout, $\pm O_p(1)$ means up to an additive error which only depends on $p$.} When $\alpha^3/(\log (1/\alpha))^{\log p} \le \epsilon < \alpha^3(1-2^{-8-8C_p})$, $n_p(\alpha,\beta)$ grows as a tower of $p$'s of height $(\log p)\log \log (1/\alpha) \pm O_p(1)$. When $\epsilon \ge \alpha^3(1-2^{-8-8C_p})$, $n_p(\alpha,\beta)$ grows as a tower of $p$'s of height $\Theta\left(\log\left(\frac{\log(1/\alpha)}{\log(\alpha^{3}/\beta)}\right)\right)\pm O_p(1)$ with a $1/\beta$ on top. This is summarized in the theorem below. We conjecture these bounds should also hold for $p<19$. The upper bounds still hold when $p<19$, as well as the lower bound when $\epsilon$ is small. The only issue is the case $\epsilon$ is large. In this case, the lower bound construction fails as it uses bounds on the largest subset of $\mathbb{F}_p^n$ without a three-term arithmetic progression, and the known bounds for this are not good enough to imply the desired estimates in this case. 

\begin{thm}
\label{main2} Let $p\geq 19$ be prime, $0<\alpha \le 1/2$, and $\alpha^3>\beta>\alpha^{3+e^{-133}}$. Recall that $n_{p}(\alpha,\beta)$
is the least positive integer such that for each $n\geq n_{p}(\alpha,\beta)$
and every subset $A$ of $\mathbb{F}_{p}^{n}$ with density at least
$\alpha$, there is a nonzero $d$ in $\mathbb{F}_{p}^{n}$ such that
the density of three-term arithmetic progressions with common difference
$d$ in $A$ is at least $\beta$. Let $\epsilon=\alpha^3-\beta$. 

\begin{itemize} 
\item If $\epsilon<\alpha^3/(\log (1/\alpha))^{\log p}$, then $n_p(\alpha,\beta)$ grows as a tower of $p$'s of height $\Theta\left(\log(\alpha^3/\epsilon)\right)\pm O_p(1)$. \item If $\alpha^3/(\log (1/\alpha))^{\log p} \le \epsilon < \alpha^3(1-2^{-8-8C_p})$, then 
$n_p(\alpha,\beta)$ grows as a tower of $p$'s of height $\Theta((\log p)\log \log (1/\alpha)) \pm O_p(1)$. \item Otherwise, $\epsilon \ge \alpha^3(1-2^{-8-8C_p})$, and  we have $n_p(\alpha,\beta)$ grows as a tower of $p$'s of height $\Theta\left((\log p)\log\left(\frac{\log(1/\alpha)}{\log(\alpha^{3}/\beta)}\right)\right)\pm O_p(1)$ with a $1/\beta$ on top. 
\end{itemize}
\end{thm}

A special case of this theorem is that we want a nonzero $d$ for
which the arithmetic progression density with common difference $d$
is at least half the random bound. In this case, for set density $\alpha$,
we get the dimension we need to guarantee such a common difference
grows as a tower of height proportional to $(\log p)\log\log(1/\alpha)$ up to an additive error depending on $p$. 
\begin{cor}
\label{cor1} The minimum dimension $n=n_{p}(\alpha,\alpha^{3}/2)$
needed to guarantee that for any subset of $\mathbb{F}_{p}^{n}$ of
density at least $\alpha$, there is a nonzero $d$ for which the density of three-term arithmetic progressions with common difference $d$ is at least half the
random bound grows as a tower of $p$'s of height $\Theta((\log p)\log\log(1/\alpha)) \pm O_p(1)$. 
\end{cor}
Corollary \ref{cor1} follows by substituting in $\beta=\alpha^3/2$ into the bound in Theorem \ref{main2}. 

If we only want to guarantee a density which is considerably smaller
than the random bound, we still get a tower-type bound by substituting in $\beta=\alpha^{3+z}$. 
\begin{cor}
\label{cor2} For $\alpha \leq 1/2$, $z<e^{-133}$ and $\alpha^z \leq 2^{-8-8C_p}$, the minimum dimension $n=n_{p}(\alpha,\alpha^{3+z})$
needed to guarantee that for any subset of $\mathbb{F}_{p}^{n}$ of
density at least $\alpha$, there is a nonzero $d$ for which the density of three-term arithmetic progressions with common difference $d$ is at least $\alpha^{3+z}$
grows as a tower of $p$'s of height $\Theta((\log p)(\log(1/z)))\pm O_p(1)$ with $1/\alpha$
on top. 
\end{cor}

The previous results only apply for $\alpha \leq 1/2$, and there is a good reason for this. The tower height for the function $n_p(\alpha,\beta)$ actually changes behavior for $\alpha$ close to one, as given by the following theorem. It determines the tower height up to an absolute constant factor and an additive term  depending on $p$. The proof uses additional ideas both in the upper and lower bounds. 
 
\begin{thm}\label{thmverydensetight}
For $\alpha \geq 1/2$, $\gamma=1-\alpha$, $\epsilon=\alpha^3-\beta$, and  $\epsilon < \gamma^2$, we have $n_p(\alpha,\beta)$ grows as a tower 
of $p$'s of height $\Theta((\log \epsilon)/(\log \gamma))\pm O_p(1)$ with $\log_p(1/\gamma)$ on top. 
\end{thm}

We remark that when $\epsilon\ge \gamma^2$, a simple argument shows that $n_p(\alpha,\beta) \le 3\log_p(1/\gamma)$ is not of tower type. 

\vspace{0.1cm}
\noindent \textbf{Organization.} We prove the tight upper and lower bounds in Theorem \ref{main2} in Sections \ref{sectionupperbound} and \ref{sectionlowerbound}, respectively. In Section \ref{verydensesection}, 
we discuss the case where the set density $\alpha$ is close to $1$, and in particular prove Theorem \ref{thmverydensetight}. In the concluding remarks we discuss many related problems and results.

For the sake of clarity of presentation, we omit floor and
ceiling signs where they are not crucial. We use $\log$ to denote the logarithm base $2$, and $\ln$ to denote the natural logarithm. We often use $3$-AP as shorthand for three-term
arithmetic progression.   

Let $\tower(a,k)$ denote a tower of $a$'s of height $k$, and $\tower(a,k,r)$ denote a tower of $a$'s of height $k$ and an $r$
on top. So $\tower(a,0)=1$ and $\tower(a,k+1)=a^{\tower(a,k)}$
for $k\geq 0$, and $\tower(a,0,r)=r$ and $\tower(a,k+1,r)=a^{\tower(a,k,r)}$
for $k\geq0$.


\section{Upper bound}

\label{sectionupperbound}

In \cite{FPI}, we have already proved that 
\[
n_{p}(\alpha,\alpha^3-\epsilon)\le\tower(p,\log((\alpha-\alpha^3)/\epsilon)+5,1/\epsilon).
\]
We next give an upper bound on $n_p(\alpha,\beta)$, which tightens the above bound when $\epsilon=\alpha^3-\beta$ is larger compared to $\alpha$. Recall that the exponent $C_p$ in the arithmetic removal lemma satisfies $C_p=\Theta(\log p)$. We assume $\beta \geq \alpha^{C_p}/2$ as otherwise we know $n_p(\alpha,\beta)=\Theta(\log(1/\alpha))$ by the discussion in the introduction. 

\begin{thm}
\label{thm:largebound} Let $\beta=\alpha^{3}-\epsilon$ and suppose that $\beta \geq \alpha^{C_p}/2$.
We have the following upper bounds. 
\begin{enumerate} 
\item 
\[
n_{p}(\alpha,\beta)\le\tower(p,\log((\alpha-\alpha^3)/\epsilon)+5,1/\epsilon).
\]

\item If $2^{-8-8C_{p}}\alpha^{3}<\beta \leq 2^{-8-8C_{p}}\alpha,$ then
\[
n_{p}(\alpha,\beta)\le\tower(p,\log (\beta/\epsilon)+O\left((\log p)\log \log_p(\alpha/\beta)\right),1/\epsilon),
\]
\item If $\beta\le 2^{-8-8C_{p}}\alpha^{3}$, then  
\[
n_{p}(\alpha,\beta)\le\tower\left(p,O\left((\log p)\log \frac{\log(\alpha/\beta)}{\log(\alpha^3/\beta)}\right),1/\beta \right).
\]
\end{enumerate}
\end{thm}

We first observe how Theorem \ref{thm:largebound} implies the claimed upper bound in Theorem \ref{main2}. 

If $\beta> 2^{-8-8C_p}\alpha$, as $\beta < \alpha^3$, we have $\alpha > 2^{-4-4C_p}$. So the bound in Theorem \ref{thm:largebound}(1) demonstrates that in this case 
$n_p(\alpha,\beta)$ is at most a tower of $p$'s of height at most $$\log((\alpha-\alpha^3)/\epsilon)+5 \leq \log(\alpha^3/\epsilon)+2\log(1/\alpha)+5 \leq \log(\alpha^3/\epsilon)+O(\log p).$$
This gives the desired bound when $\epsilon <\alpha^3/(\log(1/\alpha))^{\log p}$. When $\epsilon \ge \alpha^3/(\log(1/\alpha))^{\log p}$, since $\alpha$ is bounded below by a constant depending only on $p$, so is $\epsilon$. Hence, $\log(\alpha^3/\epsilon) \le \log(1/\epsilon) = O_p(1)$ and we obtain the desired upper bound in the remaining regions. 

So we may assume $\beta \leq 2^{-8-8C_p}\alpha$. If $\epsilon<\alpha^3/(\log 1/\alpha)^{\log p}$, as $\beta<\alpha^3$, the first term in the sum in the tower height in the bound in Theorem \ref{thm:largebound}(2) is the largest of the two terms (up to an absolute constant factor), and we get $n_p(\alpha,\beta)$ in this case is at most a tower of $p$'s of height $O(\log (\alpha^3/\epsilon))$. If $\alpha^3/(\log (1/\alpha))^{\log p}\le \epsilon < \alpha^3(1-2^{-8-8C_p})$, the second term is larger (up to a multiplicative constant and an additive term depending on $p$), and as $2^{-8-8C_p}\alpha^3<\beta<\alpha^3$, this is $O((\log p)\log \log (1/\alpha)) \pm O_p(1)$. Otherwise, we have $\epsilon \geq \alpha^3(1-2^{-8-8C_p})$ and we can apply Theorem \ref{thm:largebound}(3) to get an upper bound on $n_p(\alpha,\beta)$ which is a tower of $p$'s of height $O\left((\log p)\log \frac{\log(\alpha/\beta)}{\log(\alpha^3/\beta)}\right)$ with a $1/\beta$ on top. In any case, we get the desired upper bounds in Theorem \ref{main2}.




A $3$-AP with common difference $d$ is an ordered triple $(a,b,c)$
such that $c-b=b-a=d$. A $3$-AP is {\it trivial} if the common difference $d$ is zero, i.e., it contains
the same element three times. Otherwise, we call the $3$-AP {\it nontrivial}. 

Let $G=\mathbb{F}_{p}^{n}$. We will more generally prove the upper bounds in Theorem \ref{thm:largebound} for weighted set, given by a function $f:\mathbb{F}_p^n\rightarrow [0,1]$. For each affine subspace $H$ of $\mathbb{F}_{p}^{n}$,
let $\alpha(H)=\mathbb{E}_{x\in H}[f(x)]$ be the density of $f$
in $H$. Then $\alpha(G)=\mathbb{E}_{x\in G}[f(x)]$ is the
density of $f$. For a subspace $H$, the \textit{mean cube density} $b(H)$ is defined
to be the average of the cube of the density of $f$ in the affine
translates of $H$ which partition $\mathbb{F}_{p}^{n}$. It is also
given by $b(H)=\mathbb{\mathbb{E}}_{y\in G}[\alpha(H+y)^{3}]$, where
$H+y=\{h+y:h\in H\}$ is the affine translate of $H$ by $y$.

We define the density of $3$-APs with common difference $d$ of a weighted set $f:\mathbb{F}_{p}^{n}\to [0,1]$ as $\mathbb{E}_{x\in\mathbb{F}_p^n}[f(x)f(x+d)f(x+2d)]=\frac{1}{p^n}{\sum_{x\in\mathbb{F}_p^n}[f(x)f(x+d)f(x+2d)]}$. The density of $3$-APs with common difference $d$
of a set $A$ is the same as that of the characteristic function
of $A$. For a function $f:\mathbb{F}_{p}^{n}\rightarrow[0,1]$,
the $3$-AP density of $f$, which is $\mathbb{E}_{x,d\in\mathbb{F}_p^n}[f(x)f(x+d)f(x+2d)]$, we denote by $\Lambda(f)$. For an affine
subspace $H$, we let $\Lambda_H(f)$ denote the density of three-term
arithmetic progressions of $f$ in $H$. That is, 
\[
\Lambda_H(f)=\mathbb{E}_{x,y,z\in H,~x-2y+z=0~}[f(x)f(y)f(z)].
\]
We let $\lambda_H(f)$ denote the density of \emph{nontrivial} three-term
arithmetic progressions of $f$ in $H$. That is, 
\[
\lambda_H(f)=\mathbb{E}_{x,y,z\in H~\textrm{distinct},~x-2y+z=0~}[f(x)f(y)f(z)].
\]
As observed in \cite{FPI}, $\lambda_H(f)$ and $\Lambda_H(f)$ are close if $H$ is
large, 
\begin{equation}
\lambda_H(f)=\frac{\Lambda_H(f)\cdot|H|^{2}-|H|\cdot\mathbb{E}_{x\in H}\left[f(x)^{3}\right]}{|H|(|H|-1)}\ge\Lambda_H(f)-\frac{\mathbb{E}_{x\in H}\left[f(x)^{3}\right]}{|H|}.\label{closeLl}
\end{equation} By averaging the previous inequality over
all translates of $H$ and letting $\mathbb{E}_{x\in G}[f(x)]=\alpha$, we have 
\begin{equation}
\mathbb{E}_{y}[\lambda_{H+y}(f)]\ge\mathbb{E}_{y}[\Lambda_{H+y}(f)]-\frac{\mathbb{E}_{y\in G}\left[\mathbb{E}_{x\in H+y}\left[f(x)^{3}\right]\right]}{|H|}\ge\mathbb{E}_{y}[\Lambda_{H+y}(f)]-\frac{\alpha}{|H|}.\label{closeLlL}
\end{equation}

The proof of Theorem \ref{thm:largebound} is by a density increment
argument using the mean cube density. 

In \cite{FPI}, we proved the following lemma, which shows that if the density of $3$-APs with nonzero common
difference in a subspace $H$ is small, then the mean cube density
can be increased substantially by passing to a subspace $H'$ of bounded
codimension. 
\begin{lem}
\label{lem:mean cubed density increment_small}
If $f:\mathbb{F}_{p}^{n}\rightarrow[0,1]$ has density $\alpha$, $H$ is a subspace of $\mathbb{F}_{p}^{n}$ of size larger than $4\alpha/\epsilon$, and $\mathbb{E}_{y}[\lambda_{H+y}(f)] < \alpha^{3}-\epsilon$, then there is a subspace
$H'$ of $H$ with $\text{Codim}(H')\le\text{Codim}(H)+p^{\text{Codim}(H)}\cdot144/\epsilon^{2}$
such that $b(H')-\alpha^{3} > 2(b(H)-\alpha^{3})+\epsilon/2$.\end{lem}

Lemma \ref{lem:mean cubed density increment_large} below has a similar assumption and conclusion as the previous lemma. However, it assumes both a stronger hypothesis on the $3$-AP density, and has a stronger conclusion, that
the ratio of the mean cube density to the bound on the $3$-AP density
increases by a factor in the exponent. As before, the proof
uses the weak regularity lemma and counting lemma, but it also uses
the tight bound from \cite{FL} on the arithmetic removal lemma to
get a larger density increment. For convenience, we restate the statements of the weak regularity lemma and the counting lemma here. 

For $G=\mathbb{F}_p^n$ and $f:G \rightarrow \mathbb{C}$, define the Fourier transform $\widehat{f}(\chi)=\frac{1}{|G|}\sum_{x\in G}f(x)\chi(x)$ for characters $\chi \in \widehat{G}$. For a subspace $H$ of $G$, define the average function $f_H(x)=\mathbb{E}_{y\in H+x}[f(y)]$, which is constant on each affine translate of $H$ and has value equal to the average value of $f$ over that affine translate. A subspace $H$ is defined to be {\it $\delta$-weakly-regular} with respect to $f$ if $|\widehat{f}(\chi)-\widehat{f_H}(\chi)|\le \delta$ for all $\chi\in \widehat{G}$. Also, two functions $f,g:G\to [0,1]$ are called {\it $\delta$-close} if $|\widehat{f}(\chi)-\widehat{g}(\chi)|\le \delta$ for all $\chi\in \widehat{G}$. 
\begin{lem}
\label{lem:(Weak-regularity-lemma.)}(Weak regularity lemma.) For
any function $f:\mathbb{F}_{p}^{n}\rightarrow[0,1]$, there is a subspace
$H$ which is $\delta$-weakly-regular with respect to $f$ such that
$H$ has codimension at most $\delta^{-2}$. \end{lem}

\begin{lem}\label{lem:Counting lemma}(Counting lemma.) Suppose $f,g:\mathbb{F}_{p}^{n}\rightarrow[0,1]$
are $\delta$-close with density $\alpha$, then
$|\Lambda(f)-\Lambda(g)|\leq3\delta\alpha$.\end{lem}

As remarked in \cite{FPI}, while stated only for functions on $\mathbb{F}_p^n$, the weak regularity lemma and counting lemma can also be applied to affine subspaces of $\mathbb{F}_p^n$. 

The following density increment lemma assumes that the mean cube density is significantly larger than $\beta$, and concludes that, in passing to a large subspace, $b(H)/\beta$ increases by a power. 

\begin{lem}
\label{lem:mean cubed density increment_large} Let $f:\mathbb{F}_p^n \rightarrow [0,1]$, and 
$H$ be a subspace
of $\mathbb{F}_{p}^{n}$ with $|H| \ge 2\alpha/\beta$ and 
$b(H)\ge2^{8+8C_{p}}\beta$, where $C_{p}=\Theta(\log p)$ is the exponential constant in the
arithmetic removal lemma. If $\mathbb{E}_{y}[\lambda_{H+y}(f)] < \beta$, then there is a subspace
$H'$ of $H$ with $\text{Codim}(H')\le\text{Codim}(H)+p^{\text{Codim}(H)}\cdot36/\beta^{2}$
and $b(H')/\beta > (b(H)/\beta)^{1+\tau_{p}}$, where $\tau_p=1/(2C_p)=\Theta((\log p)^{-1})>0$.\end{lem}
\begin{proof}
Let $y=b(H)/\beta\ge 2^{8+8C_p}>2^8$ and $\eta=\beta/6$. Denote the affine translates
of $H$ by $H_{j}$, where $j\in\mathbb{F}_p^n/H$, so each affine translate of $H$ is labeled by the corresponding element in $\mathbb{F}_p^n/H$. For each
affine translate $H_{j}$ of $H$, we apply Lemma \ref{lem:(Weak-regularity-lemma.)} and Lemma \ref{lem:Counting lemma}, as remarked to apply on affine subspaces, to the translate $H_{j}$ of $H$ and the restriction of $f$ to $H_j$.  We obtain
a subspace $T_{j}$ of $H$ with
$\dim(H/T_j)\le \eta^{-2}$ such that 
\begin{equation}
|\Lambda_{H_j}(f)-\Lambda_{H_j}(t_j)|\le 3\alpha(H_j)\eta\le 3\eta \label{counting approx},
\end{equation} where $t_{j}:H_{j}\to[0,1]$ is defined by $t_j(x)=\mathbb{E}_{t\in T_j+x}[f(t)]$ for $x\in H_j$. We then let $$H'=\bigcap_{j\in \mathbb{F}_p^n/H}T_j.$$

We next prove that $$b(H')/\beta > (b(H)/\beta)^{1+\tau_p}.$$

Denote the affine translates
of $T_{j}$ in $H_{j}$ by $T_{jk}$ for $k\in H/T_j$. Since $t_j$ is constant on each translate of $T_j$ in $H_j$, we can denote by $t_j(T_{jk})$ the constant value $t_j(x)$ for $x\in T_{jk}$. Let $X_{j}$ be the set of $k\in H/T_j$ with $t_{j}(T_{jk})>\frac{\alpha(H_{j})}{y^{1/6}}$
and let $x_{j}=\frac{|X_{j}|}{|H/T_j|}$. By the arithmetic removal
lemma as discussed in the introduction, there is at least a $x_{j}^{C_{p}}$
fraction of the ordered triples $(k_1,k_2,k_3)$ that form a $3$-AP in $H/T_j$ with $k_1,k_2,k_3\in X_j$. Each $3$-AP $(k_1,k_2,k_3)$ in $H/T_j$ with $k_1,k_2,k_3\in X_j$ corresponds to three affine translates of $T_j$ in $H_j$ that form a $3$-AP of subspaces, where the value of $t_j$ on each of them is more than $\frac{\alpha(H_{j})}{y^{1/6}}$. Hence, 
\[
\Lambda_{H_j}(t_{j})\ge x_{j}^{C_{p}}\frac{\alpha(H_{j})^{3}}{y^{1/2}}.
\]
Moreover, the mean cube density over $H_{j}$ of $t_j$ is 
\begin{align}
\mathbb{E}_{k}[t_{j}(T_{jk})^{3}] & \ge (1-x_{j})\left(\frac{\alpha(H_{j})}{y^{1/6}}\right)^{3}+x_{j}\left(\alpha(H_{j})\cdot\frac{1-(1-x_{j})/y^{1/6}}{x_{j}}\right)^{3}\nonumber \\
 & =\alpha(H_{j})^{3}\left(\frac{1-x_{j}}{y^{1/2}}+\frac{(y^{1/6}-1+x_{j})^{3}}{x_{j}^{2}y^{1/2}}\right)\nonumber \\
 & =\frac{\alpha(H_{j})^{3}}{y^{1/2}}\left(1-x_{j}+\frac{(y^{1/6}-1+x_{j})^{3}}{x_{j}^{2}}\right),\label{12321}
\end{align}
where the first inequality is by Karamata's inequality (a generalization of Jensen's inequality, also known as the majorization inequality) applied to the convex function $h(x)=x^3$, and a $1-x_j$ fraction of the translates of $T_j$ have density at most $\frac{\alpha(H_{j})}{y^{1/6}}<\alpha(H_j)$. 

Since $\mathbb{E}_{j}[\lambda_{H_j}(f)] < \beta$, from (\ref{closeLlL}) and (\ref{counting approx}),
we have 
\begin{equation}
\mathbb{E}_{j}\left[x_{j}^{C_{p}}\frac{\alpha(H_{j})^{3}}{y^{1/2}}\right]\le\mathbb{E}_{j}[\Lambda_{H_j}(t_{j})]\le\mathbb{E}_{j}[\Lambda_{H_j}(f)]+3\eta\le\mathbb{E}_{j}[\lambda_{H_j}(f)]+\frac{\alpha}{|H|}+3\eta<2\beta,\label{x1y1z1}
\end{equation}
where in the last inequality we use the condition $|H|>2\alpha/\beta$
and $\eta=\beta/6$.

Let $a\in[0,1]$ be a constant to be defined later, $A$ be the set of $j$
such that $x_{j}>a$, and $I(j\in A)$ be the indicator function which
is $1$ if $j\in A$ and $0$ otherwise. From (\ref{x1y1z1}), we
have 
\[
\mathbb{E}_{j}\left[I(j\in A)\frac{\alpha(H_{j})^{3}}{y^{1/2}}\right] < \frac{2\beta}{a^{C_{p}}},
\]
and 
\[
\mathbb{E}\left[\frac{\alpha(H_{j})^{3}}{y^{1/2}}\right]=\frac{b(H)}{y^{1/2}}=\frac{y\beta}{y^{1/2}}=y^{1/2}\beta,
\]
so 
\begin{equation}
\mathbb{E}_{j}\left[I(j\notin A)\frac{\alpha(H_{j})^{3}}{y^{1/2}}\right] > y^{1/2}\beta-2a^{-C_{p}}\beta.\label{ijk}
\end{equation}
Observe that $f(x)=1-x+\frac{(z+x)^{3}}{x^{2}}=1+\frac{z^{3}}{x^{2}}+3\frac{z^{2}}{x}+3z$ is a decreasing function
in $x$ for $z>0$ and $x\in[0,1]$. Hence, when $x_j \le a$ and $z=y^{1/6}-1>0$, 
\begin{equation} \mathbb{E}_k[t_j(T_{jk})^3]\ge\frac{\alpha(H_{j})^{3}}{y^{1/2}}\left(1-x_{j}+\frac{(y^{1/6}-1+x_{j})^{3}}{x_{j}^{2}}\right) \ge\frac{\alpha(H_{j})^{3}}{y^{1/2}}\left(1-a+\frac{(y^{1/6}-1+a)^{3}}{a^{2}}\right) \label{monotone}
\end{equation}
Thus, recalling that $H'=\bigcap_{j}T_{j}$, we have 
\begin{eqnarray}
b(H') & \ge & \mathbb{E}_{j}\mathbb{E}_{k}\left[t_{j}(T_{jk})^{3}\right]\nonumber \\
 & \ge & \mathbb{E}_{j}\left[I(j\notin A)\mathbb{E}_{k}\left[t_{j}(T_{jk})^{3}\right]\right]\nonumber \\
 & \ge & \mathbb{E}_{j}\left[I(j\notin A)\frac{\alpha(H_{j})^{3}}{y^{1/2}}\cdot\left(1-a+\frac{(y^{1/6}-1+a)^{3}}{a^{2}}\right)\right]\nonumber \\
 & > & \left(1-a+\frac{(y^{1/6}-1+a)^{3}}{a^{2}}\right)\left(y^{1/2}-2a^{-C_{p}}\right)\beta\nonumber \\
 & = & \left(a^{2}-a^{3}+(y^{1/6}-1+a)^{3}\right)a^{-2}\left(y^{1/2}-2a^{-C_{p}}\right)\beta,\label{long12}
\end{eqnarray}
where the first inequality follows from Jensen's inequality applied
to the convex function $h(x)=x^{3}$, noting that the partition by
$H'$ is a refinement of the partition by $T_{j}$ in each affine
subspace $H_{j}$. The second inequality follows since the left hand
side is a sum of nonnegative terms and therefore we can delete some
of them and the sum does not increase. The third inequality is by
substituting in (\ref{12321}) and (\ref{monotone}). The
fourth inequality is by substituting in (\ref{ijk}). As $y>2^{6}$
and $a\in[0,1]$, we have 
\begin{equation}
a^{2}-a^{3}+(y^{1/6}-1+a)^{3}\ge(y^{1/6}-1)^{3}\ge y^{1/2}/8.\label{long23}
\end{equation}

Choose $a=2^{2/C_{p}}y^{-1/(2C_{p})}$, so $a \in [0,1]$
as $y>16$. It follows from (\ref{long12}) and (\ref{long23}) that
\[
b(H') > \frac{y^{1/2}}{8}a^{-2}\left(y^{1/2}-\frac{y^{1/2}}{2}\right)\beta=\frac{1}{16}a^{-2}y\beta.
\]
Recall $\tau_{p}=\frac{1}{2C_p}>0$. We have 
\[
b(H')>\frac{y^{2\tau_{p}}}{2^{4+4/C_{p}}}\cdot y\beta \geq y^{\tau_{p}}\cdot y\beta=y^{\tau_{p}}b(H).
\]
The second inequality above follows from $y \ge 2^{8+8C_{p}}=2^{\frac{4+4/C_{p}}{1/(2C_{p})}}=2^{\frac{4+4/C_{p}}{\tau_{p}}}$.
Hence, 
\[
b(H')/\beta > (b(H)/\beta)^{\tau_{p}}\cdot b(H)/\beta=(b(H)/\beta)^{1+\tau_{p}}.
\]
Moreover, we have 
\begin{align*}
\text{Codim}(H') & \le \text{Codim}(H)+\sum_{j\in \mathbb{F}_p^n/H} \dim(H/T_j) \\
& \le \text{Codim}(H)+\eta^{-2}p^{\text{Codim}(H)}\\
 & \le \text{Codim}(H)+p^{\text{Codim}(H)}\cdot36/\beta^{2}.
\end{align*}
Thus, the subspace $H'$ has the desired properties. 
\end{proof}

In \cite{FPI}, we proved the following bound on $n_p(\alpha,\alpha^3-\epsilon)$ by repeated application of Lemma 
\ref{lem:mean cubed density increment_small}. This bound is tight up to a constant factor in the tower height when $\epsilon$ is relatively small. 

\begin{thm}\label{firstuppthm}
For $\beta=\alpha^3-\epsilon$, we have $$n_p(\alpha,\beta) \leq \tower\left(p,5+\log\left(\frac{\alpha-\alpha^3}{\epsilon}\right),1/\epsilon\right).$$
\end{thm}

We next prove the main theorem in this section, Theorem \ref{thm:largebound}, by a similar proof to that of Theorem \ref{firstuppthm}. As remarked earlier, we will prove the upper bounds in the more general setting of functions $f:\mathbb{F}_p^{n}\rightarrow [0,1]$ instead of subsets.
We assume that the function $f:\mathbb{F}_{p}^{n}\rightarrow[0,1]$ has
density $\alpha$ and, for each nonzero $d$, the density of $3$-APs
with common difference $d$ is less than $\beta<\alpha^{3}$. Starting
from $H_{0}=\mathbb{F}_{p}^{n}$, we repeatedly apply Lemma \ref{lem:mean cubed density increment_small}
until we can apply Lemma \ref{lem:mean cubed density increment_large},
at each step finding a subspace of substantially larger mean cube
density at the expense of having a larger codimension. As the mean
cube density is at most $\alpha$, this yields the desired upper bound
on the dimension $n$. 

\begin{proof}[Proof of Theorem \ref{thm:largebound}.]
Theorem \ref{firstuppthm} gives the first desired bound in Theorem \ref{thm:largebound}. So we may and will assume that $\beta \leq 2^{-8-8C_p}\alpha$. 
We assume that $f:\mathbb{F}_{p}^{n}\rightarrow[0,1]$ has
density $\alpha$ and, for each nonzero $d$, the density of $3$-APs
with common difference $d$ is less than $\beta=\alpha^3-\epsilon$. Let $H_{0}=\mathbb{F}_{p}^{n}$,
so $b(H_{0})=\alpha^{3}$. We define a sequence of subspaces $H_{0}\supset H_{1} \supset \cdots \supset H_s$ recursively. Note that this implies 
$b(H_0) \leq b(H_1) \leq \ldots \leq b(H_s)$. Thus, for any $0\le i\le s$, $b(H_i)\ge b(H_0) = \alpha^3$, and $b(H_i)\le b(\{0\}) \le \alpha$ where $\{0\}$ is the trivial subspace containing only $0$.

If $b(H_i)<2^{8+8C_p}\beta$ and $|H_i| \geq 4\alpha/\epsilon$, then we apply Lemma \ref{lem:mean cubed density increment_small} to obtain a subspace 
$H_{i+1} \subset H_i$ with $$b(H_{i+1})-\alpha^3 \geq 2(b(H_i)-\alpha^3)+\epsilon/2$$ and 
\[
\text{Codim}(H_{i+1})\le\text{Codim}(H_{i})+p^{\text{Codim}(H_{i})}\cdot144/\epsilon^{2}.
\]
It follows that $2\text{Codim}(H_{i+1}) \leq \max\left(300^2\epsilon^{-4},p^{2\text{Codim}(H_{i})}\right)$. In particular, 
$\text{Codim}(H_{i+1})$ is at most a tower of $p$'s of height $i$ with a $300^2\epsilon^{-4}$ on top. As long as we applied Lemma \ref{lem:mean cubed density increment_small} to obtain $H_i$, by induction on $i$, we have 
\begin{equation} \label{firstineqb} b(H_{i}) \geq \alpha^3 +(2^i-1)\epsilon/2.
\end{equation}

Let $s_1$ be the minimum nonnegative integer $i$ for which $b(H_i) \geq 2^{8+8C_p}\beta$ or $|H_i| < 4\alpha/\epsilon$. We have $\alpha^3 \geq 2^{8+8C_p}\beta$ and $s_1 = 0$, or, by 
(\ref{firstineqb}), 
\begin{equation}\label{s1bound}
s_{1}\le\log(2^{10+8C_{p}}\beta/\epsilon) = \log(\beta/\epsilon)+\Theta(\log p).
\end{equation}
As $300^2\epsilon^{-4}<p^{p^{p^{1/\epsilon}}}$, we have 
$$\text{Codim}(H_{s_1})\le \tower(p,s_1+3,1/\epsilon).$$ 

If $|H_{s_1}|\geq 4\alpha/\epsilon$, and $|H_i|>2\alpha/\beta$ for some $i \geq s_1$, then we apply Lemma \ref{lem:mean cubed density increment_large} 
to find a subspace $H_{i+1}\subset H_{i}$ with 
\[
b(H_{i+1})/\beta\ge(b(H_{i})/\beta)^{1+\tau_{p}}
\]
and 
\[
\text{Codim}(H_{i+1})\le\text{Codim}(H_{i})+p^{\text{Codim}(H_{i})}\cdot36/\beta^{2}.
\]
Here $\tau_p =1/(2C_p)$. 
It follows that $2\text{Codim}(H_{i+1}) \leq \max\left(80^2\beta^{-4},p^{2\text{Codim}(H_{i})}\right)$. Let $s_2$ be the number of times that we apply Lemma \ref{lem:mean cubed density increment_large} before we cannot anymore. Hence, the number of subspaces $s$ we pick before stopping is $s=s_1+s_2$. 

If $\beta \leq 2^{-8-8C_p}\alpha^3$, then $s_1=0$ and we have 
$$\alpha/\beta \geq b(H_s)/\beta \geq \left(b(H_{s_1})/\beta \right)^{(1+\tau_p)^{s_2}} = (\alpha^3/\beta)^{(1+\tau_p)^{s_2}},$$
from which it follows that 
\[
s_2 \leq \log_{1+\tau_p}\left(\frac{\log (\alpha/\beta)}{\log(\alpha^3/\beta)}\right) \leq O\left((\log p)\log\left(\frac{\log (\alpha/\beta)}{\log(\alpha^3/\beta)}\right)\right).
\]
As $80^2\beta^{-4}<p^{p^{p^{1/\beta}}}$, we have 
\[
\text{Codim}(H_s) \leq \tower\left(p,s_2+3,1/\beta \right).
\] We also have $|H_s|<2\alpha/\beta$. Since $p^n=|H_s|p^{\codim(H_s)}$, we obtain that $n$ is less than $\tower(p,s_2+4,1/\beta)$. This gives the third desired bound. 
 
If $2^{-8-8C_p}\alpha^3 < \beta \leq 2^{-8-8C_p}\alpha$, we have $$\alpha/\beta \geq b(H_s)/\beta \geq \left(b(H_{s_1})/\beta \right)^{(1+\tau_p)^{s_2}} \geq 2^{(8+8C_p)(1+\tau_p)^{s_2}},$$ from which it follows that $$s_2 \leq \log_{1+\tau_p} \frac{\log (\alpha/\beta)}{8+8C_p} = \Theta\left((\log p) \log  \log_p (\alpha/\beta)\right).$$
As $80^2\beta^{-4}<80^2 \cdot 2^{32+32C_p} \alpha^{-12} < p^{p^{p^{1/\epsilon}}}$, we have $$\text{Codim}(H_s) \leq \tower\left(p,s_2+3,\max(1/\epsilon,\text{Codim}(H_{s_1}))\right) \leq \tower\left(p,s_1+s_2+6,1/\epsilon \right).$$
We also have $|H_s|<4\alpha/\epsilon$ or $|H_s|<2\alpha/\beta$. Since $p^n=|H_s|p^{\text{Codim}(H_s)}$, we obtain that $n$ is less than $\tower\left(p,s_1+s_2+7,1/\epsilon\right)$. 
This gives the second desired bound. 
\end{proof}
 
\section{Lower bound}

\label{sectionlowerbound}

In this section, a lower bound construction is given which matches the tower height in the 
upper bound from the previous section up to an absolute constant factor and an additive constant depending on the characteristic $p$.  

In \cite{FPI}, we gave a probabilistic construction which proves the following theorem. It matches the upper bound when $\epsilon$ is small compared to $\alpha$. 
\begin{thm}
\label{thm:Lower bound for line density}  \cite{FPI} For $0<\alpha\leq1/2$ and
$\epsilon\leq 2^{-161}p^{-8}\alpha^3$, there
exists $A \subset \mathbb{F}_{p}^{n}$ of density at least $\alpha$, where $n$ is a tower
of $p$'s of height at least $\frac{1}{52}\log(\alpha^{3}/\epsilon)$,
such that for all nonzero $d$ in $\mathbb{F}_{p}^{n}$, the density
of $3$-APs with common difference $d$ in $A$ is less than $\alpha^{3}-\epsilon$. 
\end{thm}

We next discuss how to obtain the lower bound in the remaining ranges in Theorem \ref{main2}. If $\epsilon>2^{-161}p^{-8}$, as $\epsilon<\alpha^3$, we have $\alpha>2^{-54}p^{-3}$, and it follows from the upper bound that $n_p(\alpha,\beta)$ is at most a constant depending only on $p$. So we may suppose $\epsilon < 2^{-161}p^{-8}$. If $\epsilon<\alpha^3/(\log 1/\alpha)^{\log p}$, then the above theorem gives the lower bound in the first case of Theorem \ref{main2}. The other case, when 
$\epsilon \geq \alpha^3/(\log 1/\alpha)^{\log p}$, we will deduce from the following theorem, which gives a lower bound in the case $\epsilon$ is large and $p \geq 19$. 

\begin{thm}
\label{thm:epsilon_large for line density} For $p\ge 19$, $0<\alpha \leq 1/2$, and $\alpha^{3+e^{-133}} \leq \beta \leq \alpha^3 \min(p^{-\log p},p^{-50})$, there exists a subset $A \subset \mathbb{F}_{p}^{n}$ of density at least $\alpha$, where $n$ is a
tower of $p$'s of height at least $\frac{1}{30}(\log p)\ln \left(\frac{\log(1/\alpha)}{\log(\alpha^{3}/\beta)}\right) $
with an $\alpha^{3}/\beta$ on top, such that for each nonzero
$d \in \mathbb{F}_{p}^{n}$, the density of $3$-APs with common difference
$d$ in $A$ is less than $\beta$. That is,  
$$n_p(\alpha,\beta) \geq \tower\left(p,\frac{1}{30}(\log p)\ln \left(\frac{\log(1/\alpha)}{\log(\alpha^{3}/\beta)}\right),\alpha^3/\beta\right).$$
\end{thm}

Note that Theorem \ref{thm:epsilon_large for line density} does not directly apply for $\beta > \alpha^3 \min(p^{-\log p},p^{-50})$. Choose a constant $C$ so that $C\ge \max(12,\frac{8+8C_p}{\log p})$, which is further independent of $p$ (recall that $C_p = \Theta(\log p)$). If $\alpha^{e^{-133}} > p^{-C\log p}$, then $\alpha$ is bounded below by a constant depending only on $p$, and from Theorem \ref{thm:largebound}, $n_p(\alpha,\beta)$ is at most a tower of $p$'s of constant height (depending on $p$). Hence, we can assume that $\alpha^{e^{-133}} \le p^{-C\log p}$ (since the bounds in Theorem \ref{main2} are up to additive constants depending on $p$). By monotonicity of $n_p(\alpha,\beta)$ in $\beta$, as $\alpha^{e^{-133}} \le p^{-C\log p}$, we can apply Theorem  \ref{thm:epsilon_large for line density}
with $\beta= \alpha^3 p^{-C\log p}\in [\alpha^{3+e^{-133}},\alpha^3\min(p^{-\log p},p^{-50})]$ to get the bound 
\begin{eqnarray*}n_p(\alpha,\beta) & \geq &  n_p(\alpha,\alpha^3 p^{-C\log p}) \geq \tower\left(p,\frac{1}{30}(\log p)\ln \left(\frac{\log (1/\alpha)}{\log p^{C\log p}} \right),p^{C\log p}\right) \\ & > & 
 \tower\left(p,\frac{1}{30}(\log p)\log \log (1/\alpha) - (\log p)\log \log p\right).\end{eqnarray*}
 We thus have the following corollary. 

\begin{cor}
\label{cor:weighted} 
For $\beta > p^{-C\log p}\alpha^3$, we have $n_p(\alpha,\beta)$ is at least a tower of $p$'s of height \\ $\frac{1}{30}(\log p)\log \log (1/\alpha) - O_p(1)$. 
\end{cor}

This corollary gives the lower bound in Theorem \ref{main2} when $ \alpha^3(1-2^{-8-8C_p}) > \epsilon \ge \alpha^3/(\log (1/\alpha))^{\log p}$ as then $\beta \ge 2^{-8-8C_p}\alpha^3 \ge p^{-C\log p}\alpha^3$. For $\epsilon \ge \alpha^3(1-2^{-8-8C_p})$, we can check that the lower bound in Theorem \ref{main2} directly follows from the bound in Corollary \ref{cor:weighted} when $ \alpha^3 (1-2^{-8-8C_p}) \le \epsilon < \alpha^3 (1-p^{-C\log p})$ and from the bound in Theorem \ref{thm:epsilon_large for line density} when $\epsilon \ge \alpha^3(1-p^{-C\log p})$. This completes the proof of Theorem \ref{main2}. Our goal for the remainder of the section is to prove Theorem \ref{thm:epsilon_large for line density}.

\subsection{From weighted to unweighted}

For the construction of the set $A$ in Theorem \ref{thm:epsilon_large for line density}, as in \cite{FPI}, 
it will be more convenient to work with a weighted set in $\mathbb{F}_{p}^{n}$,
which is given by a function $f:\mathbb{F}_{p}^{n}\to[0,1]$. 
The weighted analogue of Theorem \ref{thm:epsilon_large for line density} is given below.
Note that for the weighted constructions, it will be convenient to normalize and 
replace $\epsilon$ by $\epsilon\alpha^{3}$ and $\beta$ by $\beta\alpha^{3}$. 

\begin{thm}
\label{thm:epsilon_large} Let $p\ge19$ be a prime, $\alpha>0$, and $\alpha^{e^{-132}} \leq \beta \leq \min(p^{-\log p},p^{-100})$. 
There exists a function $f:\mathbb{F}_{p}^{n}\to[0,1]$ of density $\alpha$, where $n$
is a tower of $p$'s of height at least $\frac{1}{20}(\log p)\ln \left(\frac{\log(1/\alpha)}{\log(1/\beta)}\right)$
with a $1/\beta$ on top, such that for each nonzero $d$, the density
of $3$-APs with common difference $d$ of $f$ is less than $\beta\alpha^{3}$. 
\end{thm}

As in \cite{FPI}, we can go from the weighted version to the unweighted version by sampling. 
\begin{lem}
\label{lem:density to weighted} If $n$ is a positive integer, $p$ a prime number, $f:\mathbb{F}_{p}^{n}\to[0,1]$,
$N=p^{n}$, and $\epsilon\geq 2\left(\frac{\ln (12N)}{N}\right)^{1/2}$, then there exists a subset $A\subset\mathbb{F}_{p}^{n}$ such that
the density of $A$ and, for each nonzero $d \in \mathbb{F}_p^n$, the density of $3$-APs with common difference $d$ of $A$ deviate no more than $\epsilon$ from those of $f$. \end{lem}

Using Lemma \ref{lem:density to weighted}, Theorem \ref{thm:epsilon_large for line density} follows from Theorem \ref{thm:epsilon_large}.

\noindent {\bf Proof of Theorem \ref{thm:epsilon_large for line density}:} Apply Theorem \ref{thm:epsilon_large} with $\alpha$ and $\beta'=(\beta/\alpha^3)^2$ to obtain $n$ and $f$ satisfying the conclusion of Theorem \ref{thm:epsilon_large}. Let $\beta=\alpha^{3+z}$. In particular, $n$ is at least a tower of $p$'s of height  $$\frac{1}{20}(\log p) \ln \left(\frac{\log(1/\alpha)}{\log(1/\beta')}\right) = \frac{1}{20}(\log p) \ln \left(1/(2z)\right) \ge \frac{1}{30}(\log p)\log(1/z)$$ 
with a $1/\beta'=\alpha^{-2z}$ on top. 
By the lower bound on $n$, we have 
\[
2\left(\frac{\ln (12p^{n})}{p^{n}}\right)^{1/2}<p^{-n/3}<\alpha^{3+z}/2=\beta/6.
\]
We apply Lemma \ref{lem:density to weighted} with $\epsilon'=\beta/6$.  We obtain a set whose density is in $[\alpha-\beta/6,\alpha+\beta/6]$ and such that the density of $3$-APs for each nonzero common difference is less than $\beta'\alpha^{3}+\beta/6 <\beta/4$. Now, we simply delete or add arbitrary elements to make the set have density $\alpha$. For each nonzero $d$, the $3$-AP density with common difference $d$ can change by at most by $3\beta/6$, so the density of $3$-APs with common difference $d$ is less than $\beta/4+\beta/2< \beta$. \qed

\vspace{0.1cm}
\noindent {\bf Construction idea} 
\vspace{0.1cm}

In the next subsection, we prove Theorem \ref{thm:epsilon_large}. The general idea for the construction has some similarities to the construction we used in \cite{FPI} to prove Theorem \ref{thm:Lower bound for line density}. We partition the dimension $n=m_1+m_2+\cdots+m_s$, where $m_{i+1}$ is roughly exponential in $m_i$ for each $i$, and let $n_i=m_1+m_2+\cdots+m_i$ be the $i^{\textrm{th}}$ partial sum, so $n_1=m_1$ and $n_{i}=n_{i-1}+m_i$ for $2 \leq i \leq s$. Consider the vector space as a product of smaller vector spaces: $\mathbb{F}_p^{n} = \mathbb{F}_p^{m_1} \times \mathbb{F}_p^{m_2} \times \cdots \times \mathbb{F}_p^{m_s}$. In each step $i$, we determine a partial function $f_i:\mathbb{F}_p^{n_i} \rightarrow [0,1]$ with density $\alpha$. The function $f_i$ has the property that for each nonzero $d \in \mathbb{F}_p^{n_i}$, the density of $3$-APs with common difference $d$ of $f_i$ is less than $\beta\alpha^3$. 

For Theorem \ref{thm:epsilon_large}, we let $m_1=\left\lfloor 10000\log_{p}(1/\beta)\right\rfloor$ and  choose $f_1$ to be the characteristic function, appropriately scaled to have average value $\alpha$ on $\mathbb{F}_p^{m_1}$, of a maximum subset of $\mathbb{F}_p^{m_1}$ with no $3$-AP. 

For $i \geq 2$, observe that we can use $f_{i-1}$ to define a function $g_i:\mathbb{F}_p^{n_{i}} \rightarrow [0,1] $ by letting $g_i(x)=f_{i-1}(y)$, where $y$ is the first $n_{i-1}$ coordinates of $x$. Thus, $g_i$ has constant value $f_{i-1}(y)$ on the copy of $\mathbb{F}_p^{m_i}$ consisting of those elements of $\mathbb{F}_p^{n_{i}}$ whose first $n_{i-1}$ coordinates equal $y$. We perturb $g_i$ to obtain $f_{i}$  so that it has several useful properties. 

We first describe some of the useful properties $f_{i}$ will have.  While $g_i$ has constant value $f_{i-1}(y)$ on each copy of $\mathbb{F}_p^{m_i}$ whose first $n_{i-1}$ coordinates are equal to $y$, the function $f_{i}$ will not have this property, but will still have average value $f_{i-1}(y)$ on each of these copies. Another useful property is that for each $d \in \mathbb{F}_p^{n_i}$ such that $d$ is not identically $0$ on the first $n_{i-1}$ coordinates, the density of $3$-APs with common difference $d$ in $f_{i}$ is equal to the density of $3$-APs with common difference $d^{*}$ in $f_{i-1}$, where $d^{*} \in \mathbb{F}_p^{n_{i-1}} \setminus \{0\}$ is the first $n_{i-1}$ coordinates of $d$. Once we have established this property, it suffices then to check that for each nonzero $d \in \mathbb{F}_p^{n_i}$ with the first $n_{i-1}$ coordinates of $d$ equal to $0$, the density of $3$-APs with common difference $d$ is less than $\beta \alpha^3$. In order to check this, it now makes sense to explain a little more about how we obtain $f_i$ from $g_i$. 

Consider a set $B \subset \mathbb{F}_p^{m_{i}}$ with relatively few three-term arithmetic progressions (considerably less than the random bound) given its size. In \cite{FPI}, we took $B$ to be the elements whose first coordinate is in an interval of length roughly $2p/3$ in $\mathbb{F}_p$. Here, we take $B$ to be the elements whose first $r_i$ coordinates (for an appropriately chosen $r_i$) are in a maximum subset of $\mathbb{F}_p^{r_i}$ with no $3$-AP. 

We consider the $p^{n_{i-1}}$ copies of $\mathbb{F}_p^{m_{i}}$ in $\mathbb{F}_p^{n_{i}}$, where each copy has the first $n_{i-1}$ coordinates fixed to some $y \in \mathbb{F}_p^{n_{i-1}}$. For each copy $A$, we consider a random copy of $B$ in $A$ by taken a random linear transformation of full rank from $\mathbb{F}_p^{m_{i}}$ to $A$ and consider the image of $B$ by this linear transformation, and then scale the indicator function of the image of $B$ by the constant factor $p^{m_i}/|B|$ to keep the average density unchanged on $A$. We do this independently for each $A$ where $g_i$ is nonzero on $A$. We show that with high probability, for every nonzero $d \in \mathbb{F}_p^{n_{i}}$ with the first $n_{i-1}$ coordinates of $d$ equal to $0$, the density of $3$-APs with common difference $d$ is less than $\beta \alpha^3$. One can show this for each such $d$ by observing that the density of $3$-APs with common difference $d$ is just the average of the densities of $3$-APs with common difference $d$ on each of the $p^{n_{i-1}}$ copies of $\mathbb{F}_p^{m_i}$. The densities of $3$-APs with common difference $d$ in the perturbed subspaces are independent random variables that have expected value (appropriately scaled) equal to the density of $3$-APs in $B$, which is much less than the random bound for a set of this size. We can then use Hoeffding's inequality, which allows us to show that the sum of a set of independent random variables with values in $[0,1]$ is highly concentrated on its mean, to show that it is very unlikely that the density of $3$-APs with common difference $d$ is at least $\beta\alpha^3$. Since the probability is so tiny, a simple union bound allows us to get this to hold simultaneously for all nonzero $d$. This completes the construction idea. 

To compare with the construction in \cite{FPI}, there, we perturb only a sufficient fraction of the affine subspaces. Here, to account for the large decrease in $3$-AP density we need to make in each step (as $\epsilon$ is large), we need to perturb all subspaces, and additionally use a perturbation that has significantly smaller $3$-AP density. In the next subsection we present the construction of a set that has significantly few three-term arithmetic progressions, which serves as a main ingredient in our construction. 


\vspace{0.2cm}

\subsection{Subsets with few arithmetic progressions} 
\label{relativelyfewsubs}
\vspace{0.2cm}

An important ingredient in our constructions is subsets of $\mathbb{F}_p^n$ with significantly few arithmetic progressions. For this purpose, we will use a set with no nontrivial $3$-AP. We let $A_{p,n}$ be a subset of
$\mathbb{F}_{p}^{n}$ of maximum size with no nontrivial $3$-AP. Let $r(p,n)=|A_{p,n}|$. Alon, Shpilka, and Umans \cite{ASU} gave a construction of a subset of 
$(\mathbb{Z}/p\mathbb{Z})^n$ of cardinality $\lceil p/2 \rceil^{n-o(n)}$ with no $3$-AP. More recently, Alon \cite{Alon} observed that a variant of the Behrend construction gives an even better bound of the same form. We present Alon's construction in the proof of the following lemma. 
\begin{lem}[Alon \cite{Alon}]\label{Alonlem}
We have $r(p,n)\ge \left(\frac{p+1}{2}\right)^{n}/\left(1+n\left(\frac{p-1}{2}\right)^{2}\right)$.\end{lem}
\begin{proof}
Consider the subset of $\mathbb{F}_{p}^{n}$ of all points with coordinates
in $\{0,1,...,\frac{p-1}{2}\}$.

For each point $z\in\{0,1,...,\frac{p-1}{2}\}^{n}$, let $C(z)=\sum_{i=1}^{n}z_{i}^{2}$. For all $z\in\{0,1,...,\frac{p-1}{2}\}^{n}$, we have $C(z)\in\left[0,n\left(\frac{p-1}{2}\right)^{2}\right]$. Hence, there exists a value $c$ such that $|\{z|C(z)=c\}|\ge\frac{\left(\frac{p+1}{2}\right)^{n}}{1+n\left(\frac{p-1}{2}\right)^{2}}$. Let $A=\{z|C(z)=c\}$. When viewed as a subset of $\mathbb{R}^d$, set $A$ is a subset of the sphere centered at the origin and with radius $\sqrt{c}$. In $\mathbb{R}^d$, any line intersects the boundary of a convex set (such as a sphere) in at most two points and hence cannot contain a $3$-AP. So $A$ has no $3$-AP when viewed as a subset of $\mathbb{R}^d$. As the coordinates of the points in $A$ have value between $0$ and $\frac{p-1}{2}$, there is no wrap-around when adding two elements of $A$, and it follows that any $3$-AP in $A$ must also be a $3$-AP in $\mathbb{R}^d$. Hence, $A \subset \mathbb{F}_p^n$ has no $3$-AP and has the desired size.  
\end{proof}

The bound in Lemma \ref{Alonlem} for $p \geq 19$ and $n \geq 10^6$ gives the bounds in the following corollary. 

\begin{cor}
\label{lem:free-set-large-n} If $p \ge 19$ is prime and $n \ge 10^6$, then $r(p,n)\ge p^{.782n}$ and $r(p,n) \geq ((p+1)/2.0001)^{n-2}$.\end{cor}

\subsection{Proof of Theorem \ref{thm:epsilon_large}}

\subsubsection{The Construction}

{\bf Choice of constants.} Let $s=\left\lfloor \frac{1}{11}(\log p)\ln \left(\frac{\log (1/\alpha)}{\log (1/\beta)}\right)-4\log p\right\rfloor$. Observe that $s \geq \left\lfloor 8\log p \right\rfloor > 33$ as $\beta \geq \alpha^{e^{-132}}$ and $p \geq 19$. In particular, $s \geq  \frac{1}{20}(\log p)\ln \left(\frac{\log (1/\alpha)}{\log (1/\beta)}\right)+3$. 

 We soon recursively define sequences of positive integers $m_{1},...,m_{s}$ and $r_1,\ldots,r_s$. We let $n=\sum_{i=1}^{s}m_{i}$, and $n_{j}=\sum_{i=1}^{j}m_{i}$ and $u_j=\sum_{i=1}^j r_i$ be the $j$th partial sums. Let $N_j=p^{n_j}$, $R_j=p^{r_j}$, and $U_j=p^{u_j}$, so in these cases we use capital letters to denote $p$ raised to the lower case power. We recall that $A_{p,r}$ is a maximum subset of $\mathbb{F}_p^r$ that contains no $3$-AP. Let $\mu_j=\prod_{i=1}^j \left(|A_{p,r_i}|/R_i\right)$. 

Let $m_{1}=\left\lfloor 10000\log_{p}(1/\beta)\right\rfloor$, so $p^{-(1+m_1)/10000} < \beta \leq p^{-m_{1}/10000}$. Note that as $\beta\le p^{-100}$, we have $m_{1}\ge 10^6$. Let $\gamma = 3/\log (p/8)$. 
Let $r_{1}=m_{1}$, $r_{2}=.7864r_{1}$, $r_3=6r_1$, and $r_i = (1+\gamma)^{i-3}r_3$ for $i \geq 4$. 
For $i\ge 2$, let $m_{i} = N_{i-1} \left(|A_{p,r_{i}}|/R_{i}\right)^2 \mu_{i-1}^{3} \beta$. 

The constants above were carefully chosen so as to work in the case $p=19$. For larger values of $p$, there is more flexibility in the choices of the constants. We collect several useful bounds between the constants in Appendix \ref{appenbounds}.

\vspace{0.1cm}
 
We divide the construction process into levels in order, starting
with level $1$ and ending at level $s$.
\vspace{0.1cm}
 
\noindent \textbf{Construction for level 1. }Let $f_1:\mathbb{F}_p^{n_1} \rightarrow [0,1]$ be defined by $f_{1}(a)= \alpha/\mu_1$ for $a \in A_{p,n_1}$ and $f_1(a)=0$ otherwise. The density of $f_1$ is $\alpha$. 

\vspace{0.1cm}

\noindent \textbf{Construction for level $i$ for $2\le i\le s$.} At level
$i$, we have already constructed a function $f_{i-1}:\mathbb{F}_{p}^{n_{i-1}}\to[0,1]$ at the previous level such that the density of 3-APs with common difference $d$ of $f_{i-1}$ is less than $\beta \alpha^3$ for each nonzero $d$ in $\mathbb{F}_{p}^{n_{i-1}}$. 


For each $x\in \mathbb{F}_p^{n_{i-1}}$, choose vectors $v_{1}(x),...,v_{r_{i}}(x)\in\mathbb{F}_{p}^{m_{i}}$
such that for any distinct $a,b,c \in \mathbb{F}_p^{n_{i-1}}$, the collection $\bigcup_{j=1}^{r_i} \{v_{j}(a),v_{j}(b),v_{j}(c)\}$ of 
$3r_i$ vectors are linearly independent, and, for each codimension one affine subspace $S$ of $\mathbb{F}_p^{m_i}$, the number of $x \in \mathbb{F}_p^{n_{i-1}}$ with $f_{i-1}(x) \not = 0$ for which $\{v_j(x): 1 \leq j \leq r_i\} \subset S$ is less than $m_i$. The existence of such a choice of vectors is guaranteed by Lemma \ref{lem:Choice of random directions-1} below.

For each $a\in\mathbb{F}_{p}^{n_{i}}$, with $x\in\mathbb{F}_{p}^{n_{i-1}}$
being the first $n_{i-1}$ coordinates of $a$ and $y\in\mathbb{F}_{p}^{m_{i}}$
the last $m_{i}$ coordinates of $a$, let $z_{a}=\left(y\cdot v_{1}(x),...,y\cdot v_{r_{i}}(x)\right)\in\mathbb{F}_{p}^{r_{i}}$.
Define $f_{i}:\mathbb{F}_{p}^{n_{i}}\to[0,1]$ such that $f_{i}(a)=\frac{R_i}{|A_{p,r_{i}}|}\cdot f_{i-1}(x)$
if $z_{a} \in A_{p,r_{i}}$ and $f_i(a)=0$ otherwise. This completes the construction for level $i$. 

When we have finished the construction for level $s$, let $f=f_{s}$.
It is clear from the construction that the density of each $f_{i}$
is $\alpha$, and there are $X_{i}:=N_i\prod_{j=1}^{i}\left(|A_{p,r_{j}}|/R_j\right)$
elements $a \in \mathbb{F}_{p}^{n_{i}}$ for which $f_{i}(a) \not = 0$. Indeed, 
at step $i$, the fraction of points in $\mathbb{F}_p^{n_i}$ where $f_i$ is nonzero is a proportion $|A_{p,r_i}| / R_i$ 
of the fraction of points in $\mathbb{F}_p^{n_{i-1}}$ where $f_{i-1}$ is nonzero. Moreover, all the nonzero values of $f_{i}$ are the same, and we denote this value by $\alpha_{i}=\alpha \prod_{j=1}^{i}\left(R_j/|A_{p,r_{j}}|\right)$.
Note that $\alpha_i=\mu_i^{-1}\alpha$ and $X_{i}=\mu_{i}N_i$.

\subsubsection{The proof}

The following lemma shows that we can choose the vectors $v_{j}(x)$
as specified in the above construction. 
\begin{lem}
\label{lem:Choice of random directions-1} For each $i\ge2$ there exists
a choice of $v_{j}(x) \in \mathbb{F}_p^{m_i}$ for $1\le j\le r_{i}$ and $x\in \mathbb{F}_p^{n_{i-1}}$ such that for any distinct $a,b,c\in\mathbb{F}_{p}^{n_{i-1}}$, the collection $\bigcup_{j=1}^{r_i} \{v_{j}(a),v_{j}(b),v_{j}(c)\}$ of $3r_i$ vectors are linearly independent, and for each codimension one affine subspace $S$ of $\mathbb{F}_p^{m_i}$,  the number of $x \in \mathbb{F}_p^{n_{i-1}}$ with $f_{i-1}(x) \not = 0$ for which $\{v_j(x): 1 \leq j \leq r_i\} \subset S$ is less than $m_i$. 
\end{lem}
\begin{proof}
For $1 \leq j \leq r_i$ and $x \in \mathbb{F}_p^{n_{i-1}}$, choose a random vector 
$v_{j}(x)\in \mathbb{F}_p^{m_i}$  uniformly so that the choices of $v_{j}(x)$ are independent. Consider distinct $a,b,c \in \mathbb{F}_p^{n_{i-1}}$. Label the $3r_i$ vectors in $\bigcup_{j=1}^{r_i} \{v_{j}(a),v_{j}(b),v_{j}(c)\}$ 
as $v_1,v_2,\ldots,v_{3r_i}$. 
The probability that $v_1$ is nonzero is $1-p^{-m_i}$. For $2 \leq j \leq 3r_i$, the probability that $v_{j}$ 
is not in the span of $\{v_{1},...,v_{j-1}\}$ given that $\{v_{1},...,v_{j-1}\}$
are linearly independent is $1-p^{-m_{i}+j-1}$. Hence, the
probability that $\bigcup_{j=1}^{r_i} \{v_{j}(a),v_{j}(b),v_{j}(c)\}$ are linearly independent
is 
\begin{eqnarray*}\prod_{k=0}^{3r_{i}-1}\left(1-p^{-m_{i}+k}\right) & \ge & 1-\sum_{k=0}^{3r_i - 1} p^{-m_{i}+k} \ge 1 - \frac{p}{p-1}p^{-m_i+3r_i-1} \ge 1- p^{-m_i+3r_i},  \end{eqnarray*}
where in the first inequality we repeatedly used the inequality $(1-y)(1-z)>1-(y+z)$ for $y,z$ nonnegative, and in the second inequality we bounded a finite geometric series by the infinite one. Thus, by the union bound, the probability that for some distinct $a,b,c \in \mathbb{F}_p^{n_{i-1}}$, the vectors in the set $\bigcup_{j=1}^{r_i} \{v_{j}(a),v_{j}(b),v_{j}(c)\}$ are not linearly independent is at most 
\begin{equation}
N_{i-1}^3p^{-m_i+3r_i} = p^{3n_{i-1}+3r_i - m_i} \leq p^{-m_i/4} \leq p^{-1} \leq 1/3,\label{eq:linear-independence}
\end{equation}
where in the first inequality we used $m_i \geq 8n_{i-1}$ and $m_i \geq 8r_i$, which follow from (\ref{sim1}) and (\ref{sim2}).

Let $S$ be a codimension one affine subspace of $\mathbb{F}_p^{m_i}$. For each $x \in \mathbb{F}_p^{n_{i-1}}$, the probability that the set $\{v_j(x): 1 \leq j \leq r_i\}$ is a subset of $S$ is $p^{-r_i}$. By the union bound, the probability that the number of $x \in \mathbb{F}_p^{n_{i-1}}$ with $f_{i-1}(x) \not =0$ for which $\{v_j(x): 1 \leq j \leq r_i\} \subset S$ is at least $m_i$ is at most 
${X_{i-1} \choose m_i}p^{-r_im_i}$. The number of codimension one affine subspaces of $\mathbb{F}_p^{m_i}$ is $p\frac{p^{m_i}-1}{p-1}<2p^{m_i}$. 
Thus, by the union bound, the probability that there is a codimension one affine subspace $S$ of $\mathbb{F}_p^{m_i}$ for which at least $m_i$ of the $x \in \mathbb{F}_p^{n_{i-1}}$ with $f_{i-1}(x) \not =0$ satisfy $\{v_j(x): 1 \leq j \leq r_i\} \subset S$ is at most 
\begin{equation}\label{lastestimatefor} 2p^{m_i}{X_{i-1} \choose m_i}p^{-r_im_i} \leq 2p^{m_i}\left(\frac{X_{i-1}e}{m_i}\right)^{m_i} p^{-r_im_i} = 2p^{m_i}
\left(\frac{\mu_{i-1}N_{i-1}e}{m_iR_i}\right)^{m_i}<2p^{m_i-5m_i}<1/3,\end{equation}
where the second to last inequality is by (\ref{sim6}). 

As $1/3+1/3 < 1$, with positive probability (and hence there exists an instance such that) for all distinct  $a,b,c \in \mathbb{F}_p^{n_{i-1}}$, the vectors in the set $\bigcup_{j=1}^{r_i}\{v_{j}(a),v_{j}(b),v_{j}(c)\}$ are linearly independent, and no codimension one affine subspace $S$ of $\mathbb{F}_p^{m_i}$ is such that $m_i$ of the $x \in \mathbb{F}_p^{n_{i-1}}$ with $f_{i-1}(x) \not =0$ satisfy $\{v_j(x): 1 \leq j \leq r_i\} \subset S$. This completes the proof of the lemma. 
\end{proof}

We now prove Theorem \ref{thm:epsilon_large}. 
\begin{proof}[Proof of Theorem \ref{thm:epsilon_large}.]

We need to prove that for each $i$, $1 \leq i \leq s$, the following holds. For every nonzero $d \in \mathbb{F}_p^{n_i}$, the density of $3$-APs with common difference $d$ of $f_i$ is less than $\beta \alpha^3$. We will prove this by induction on $i$. 

The base case $i=1$ follows from the fact that the set where $f_1$ is nonzero has no nontrivial $3$-AP. Assume that for $i=k$ and every nonzero $d \in \mathbb{F}_p^{n_i}$, the density of $3$-APs with common difference 
$d$ of $f_i$ is less than $\beta \alpha^3$. We next show this for level $i=k+1$. 

Let $\rho_{j}(d)$ be the density of $3$-APs with common difference $d$ of $f_j$.

If a nonzero $d \in \mathbb{F}_p^{n_{k+1}}$ is $0$ in the first $n_{k}$
coordinates, let the last $m_{k+1}$ coordinates
of $d$ be $d'$, which must be nonzero. If $d'$ is perpendicular
to all $v_{j}(x),1\le j\le r_{i}$, for some $x$ where $f_{k}(x)\ne0$,
then the density of $3$-APs with common difference $d$ inside the affine subspace of points where the first $n_{k}$
coordinates are equal to $x$ is $\left(|A_{p,r_{k+1}}|/R_{k+1}\right)\alpha_{k+1}^3=\left(R_{k+1}/|A_{p,r_{k+1}}|\right)^2 \alpha_k^3$;
otherwise this density is $0$, since $A_{p,r_{k+1}}$ is $3$-AP free. 

Let $t$ be the number of $x$ in $\mathbb{F}_{p}^{n_{k}}$ with $f_{k}(x)\ne0$
such that $d'\cdot v_{j}(x)=0$ for all $1\le j\le r_{i}$. The density
of $3$-APs with common difference $d$ is
\[
\rho_{k+1}(d)=\frac{t}{N_k}\cdot\left(R_{k+1}/|A_{p,r_{k+1}}|\right)^2 \alpha_k^3=\frac{t}{N_k}\cdot\left(R_{k+1}/|A_{p,r_{k+1}}|\right)^2 \cdot \mu_k^{-3} \alpha^3,
\]
where the second equality is by $\alpha_{k}=\mu_k^{-1}\alpha$. Since $d'\ne0$, we have $t < m_{k+1}$, as the construction requires that the number of $x\in \mathbb{F}_{p}^{n_{k}}$ with $f_k(x)\ne 0$ such that $v_{j}(x),1\le j\le r_{k+1}$ are contained in the orthogonal complement of $d'$ (which has codimension one) is less than $m_{k+1}$, which is guaranteed by Lemma \ref{lem:Choice of random directions-1}. Hence, as $m_{k+1}=N_{k} \left(|A_{p,r_{k+1}}|/R_{k+1}\right)^2\cdot\mu_k^3 \cdot \beta$, we have $\rho_{k+1}(d) < \beta\alpha^{3}$.

If $d \in \mathbb{F}_p^{n_{k+1}}$ is nonzero in the first $n_{k}$
coordinates, letting $d^*$ denote the first $n_k$ coordinates of $d$, by Lemma \ref{lem:Stability-2} below, we have $\rho_{k+1}(d)=\rho_{k}(d^*)<\beta\alpha^{3}$. 
We have thus proved by induction that for each nonzero $d \in \mathbb{F}_p^{n_i}$, the density of $3$-APs with common difference $d$ of $f_i$ is less than $\beta \alpha^3$. 

We chose $s$ to ensure that the functions $f_{i}$ take values $\alpha_i\in [0,1]$. Indeed, 
\begin{eqnarray*} 0 & < & \alpha \leq \alpha_i \leq \alpha_{s}=\mu_s^{-1}\alpha \leq p^{2s}2.0001^{u_s}\alpha \leq  p^{2s}2.0001^{(1+\gamma^{-1})r_s}\alpha = p^{2s}2.0001^{(\log p)r_s/3}\alpha \\ & \leq & p^{r_s/2} \alpha = p^{(1+\gamma)^{s-3} 3m_1} \alpha \leq p^{e^{\gamma s} 3m_1} \alpha\leq  \beta^{-3 \cdot 10^4 \cdot e^{\gamma s}} \alpha < 1, 
\end{eqnarray*}
where we used Corollary \ref{lem:free-set-large-n} to get $|A_{p,r_i}| \geq p^{-2}2.0001^{-r_i}|R_i|$ and hence the fourth inequality, Inequality (\ref{similar}) to get the fifth inequality, the next equality is by the choice of $\gamma=3/\log(p/8)$, the next inequality is by Inequality (\ref{rilowerbound}), the next equality is by the choice of $r_s$, the second to last inequality is by $m_1 \leq 10000\log_p (1/\beta)$, and the last inequality is by the choice of $\gamma$ and $s$ so that $$\gamma s \leq \frac{3}{\log(p/8)} \cdot \left(\frac{1}{11}(\log p)\ln \left( \frac{\log(1/\alpha)}{\log(1/\beta)}\right) -4\log p\right)< \ln \left(\frac{\log(1/\alpha)}{\log(1/\beta)}\right)-12$$
and hence 
$e^{\gamma s} < e^{-12}\left(\frac{\log(1/\alpha)}{\log(1/\beta)}\right)<10^{-5}\left(\frac{\log(1/\alpha)}{\log(1/\beta)}\right)$. 

Now, we estimate the dimension $n=n_s$ of our final space. By the definition, we have $n_s \geq m_s$. By (\ref{sim4}), we have $m_s$ is at least a tower of $p$'s of height $s-3$ with a $m_3$ on top.

In (\ref{sim23}), we observed that $m_2 \geq p^{m_1/400} \geq p^{25\log_p(1/\beta) -1} \geq 1/\beta$. By (\ref{sim1}), we have $m_3>m_2>1/\beta$. Hence, 
the dimension $n$ of our final space is at least a tower of $p$'s of height at least $s-3 > \frac{1}{20}(\log p)\ln\left(\frac{\log (1/\alpha)}{\log (1/\beta)}\right)$ with a $1/\beta$ on top. 
\end{proof}

We now provide the proof of Lemma \ref{lem:Stability-2}, which is very similar to the proof of Lemma 12 in \cite{FPI} but requires some modifications. We include it here for completeness. 
\begin{lem}
\label{lem:Stability-2}
Let $2 \leq i \leq s$. For $d \in \mathbb{F}_p^{n_i}$, let $d^{*} \in \mathbb{F}_p^{n_{i-1}}$ be the first $n_{i-1}$ coordinates of $d$. If $d^{*}$ is nonzero,  then the density of $3$-APs with common difference $d$ of $f_i$ is the same as the density of $3$-APs with common difference $d^{*}$ of $f_{i-1}$. That is, if $d^{*}$ is nonzero, then $\rho_{i}(d)=\rho_{i-1}(d^{*})$.
\end{lem}
\begin{proof}
Since $d^*$ is nonzero, for any $3$-AP $a,b,c$ with common difference $d$, the restrictions of
$a,b,c$ to the first
$n_{i-1}$ coordinates are distinct. Let $a^{*}$ be the first
$n_{i-1}$ coordinates of $a$. Similarly define
$b^{*},c^{*},d^{*}$. Fix $a^{*}=a_{0},b^{*}=b_{0},c^{*}=c_{0}$ for any $3$-AP $(a_0,b_0,c_0)$ in $\mathbb{F}_p^{n_{i-1}}$ with common difference $d^{*}$, and consider all $3$-APs $a,b,c$ with common difference $d$ such that the first
$n_{i-1}$ coordinates of $a,b,c$ coincide
with $a_{0},b_{0},c_{0}$. Since $a_{0},b_{0},c_{0}$ are distinct, the $3r_i$ vectors 
$v_{j}(a_{0})$, $v_{j}(b_{0})$, $v_{j}(c_{0})$ with $1 \leq j \leq r_i$ are linearly independent.
Hence, we can change the basis and view $v_{j}(a_{0})$, $v_{j}(b_{0})$, $v_{j}(c_{0})$ for $1 \leq j \leq r_i$
as basis vectors of $\mathbb{F}_p^{m_i}$. Let $B$ be a basis of $\mathbb{F}_p^{m_i}$ containing $v_{j}(a_{0})$, $v_{j}(b_{0})$, $v_{j}(c_{0})$ for $1 \leq j \leq r_i$. Let $a',b',c',d'$ be the restriction of $a,b,c,d$
to the last $m_{i}$ coordinates. 

For $1\le j\le r_i$, let $a_{1j},b_{1j},c_{1j}$
be any three fixed values in $\mathbb{F}_{p}$. Let $L=p^{m_i-3r_i}$. We prove that the number of $3$-APs in $\mathbb{F}_p^{n_i}$ with common difference $d$, $a^{*}=a_{0}$
and $a'\cdot v_j(a_{0})=a_{1j}$, $b^{*}=b_{0}$ and $b'\cdot v_j(b_0)=b_{1j}$, $c^{*}=c_0$ and $c'\cdot v_j(c_0)=c_{1j}$ is $L$. Since $B$ is a basis for $\mathbb{F}_p^{m_i}$, if $x,x'\in \mathbb{F}_p^{m_i}$ satisfy $x\cdot v=x'\cdot v$ for all $v\in B$ then by linearity $x\cdot v=x'\cdot v$ for all $v\in \mathbb{F}_p^{m_i}$, in which case $x=x'$. Since there are $p^{m_i}$ elements in $\mathbb{F}_p^{m_i}$, and $p^{m_i}$ possible tuples $(x\cdot v)_{v\in B}$, each tuple must appear exactly once. The $3$-APs with common difference $d$, $a^{*}=a_{0}$
and $a'\cdot v_j(a_{0})=a_{1j}$, $b^{*}=b_{0}$ and $b'\cdot v_j(b_0)=b_{1j}$, $c^{*}=c_0$ and $c'\cdot v_j(c_0)=c_{1j}$ are given by triples $(a',a'+d',a'+2d')$ such that $a'\cdot v_j(a_{0})=a_{1j}$, $a'\cdot v_j(b_0) = b_{1j}-d'\cdot v_j(b_0)$, $a'\cdot v_j(c_0)=c_{1j}-2d'\cdot v_j(c_0)$. Hence the number of such $3$-APs is equal to the number of $a'\in \mathbb{F}_p^{m_i}$ such that $a'\cdot v_j(a_{0})=a_{1j}$, $a'\cdot v_j(b_0) = b_{1j}-d'\cdot v_j(b_0)$, $a'\cdot v_j(c_0)=c_{1j}-2d'\cdot v_j(c_0)$. There is exactly one element of $\mathbb{F}_p^{m_i}$ such that its dot product with each vector in the basis $B \supset \{v_j(a_0),v_j(b_0),v_j(c_0),1\le j\le r_i\}$ is fixed to an arbitrary value. Since there are exactly $p^{m_i-3r_i}=L$ ways to choose the value of $a'\cdot v$ for $v\in B \setminus \{v_j(a_0),v_j(b_0),v_j(c_0),1\le j\le r_i\}$, there are exactly $L$ elements $a'\in \mathbb{F}_p^{m_i}$ such that $a'\cdot v_j(b_0) = b_{1j}-d'\cdot v_j(b_0)$, $a'\cdot v_j(c_0)=c_{1j}-2d'\cdot v_j(c_0)$. Hence, the number of $3$-APs $(a,b,c)$ with common difference $d$, $a^{*}=a_{0}$
and $a'\cdot v_j(a_{0})=a_{1j}$, $b^{*}=b_{0}$ and $b'\cdot v_j(b_0)=b_{1j}$, $c^{*}=c_0$ and $c'\cdot v_j(c_0)=c_{1j}$ is exactly $L$. 



For $x \in \mathbb{F}_p^{n_{i-1}}$, by definition of $f_i$, we can define $i_{x}:\mathbb{F}_{p}^{r_{i}}\to\mathbb{R}$ such that $f_i(a)=f_{i-1}(a^*)+i_{a_0}(a'\cdot v_1(a_0),...,a'\cdot v_{r_i}(a_0))$. Moreover $\mathbb{E}_{w\in \mathbb{F}_{p}^{r_{i}}}[i_{x}(w)]=0$.
The density of $3$-APs $(a,b,c)$ with common difference $d$ of $f_i$ such that $a^*=a_{0},b^*=b_{0},c^*=c_{0}$ is 

\begin{align*}
 & \frac{1}{p^{3r_{i}}L}\sum_{a_{1j},b_{1j},c_{1j}\in\mathbb{F}_{p}}L \cdot (f_{i-1}(a_{0})+i_{a_{0}}(a_{11},...,a_{1r_{i}}))\cdot(f_{i-1}(b_{0})+i_{b_{0}}(b_{11},...,b_{1r_{i}}))\cdot(f_{i-1}(c_{0})+i_{c_{0}}(c_{11},...,c_{1r_{i}}))\\
 & =\frac{1}{p^{3r_{i}}}\sum_{a_{1j},b_{1j},c_{1j}\in\mathbb{F}_{p}}(f_{i-1}(a_{0})+i_{a_{0}}(a_{11},...,a_{1r_{i}}))\cdot(f_{i-1}(b_{0})+i_{b_{0}}(b_{11},...,b_{1r_{i}}))\cdot(f_{i-1}(c_{0})+i_{c_{0}}(c_{11},...,c_{1r_{i}}))\\
 & =\frac{1}{p^{3r_{i}}}\left(\sum_{a_{1j}\in\mathbb{F}_{p}}(f_{i-1}(a_{0})+i_{a_{0}}(a_{11},...,a_{1r_{i}}))\right)\left(\sum_{b_{1j}\in\mathbb{F}_{p}}(f_{i-1}(b_{0})+i_{b_{0}}(b_{11},...,b_{1r_{i}}))\right)\cdot\\
 & \qquad\qquad\qquad\qquad\qquad\qquad\qquad\qquad\cdot\left(\sum_{c_{1j}\in\mathbb{F}_{p}}(f_{i-1}(c_{0})+i_{c_{0}}(c_{11},...,c_{1r_{i}}))\right)\\
 & =f_{i-1}(a_{0})f_{i-1}(b_{0})f_{i-1}(c_{0}).
\end{align*}

Hence,

\begin{align*}
\rho_{i}(d) & =\mathbb{E}_{a_{0},b_{0},c_{0}}[f_{i-1}(a_{0})f_{i-1}(b_{0})f_{i-1}(c_{0})]=\rho_{i-1}(d^*),
\end{align*}
which completes the proof.
\end{proof}

\section{Popular differences in very dense sets}\label{verydensesection}

In Theorem \ref{main2}, we proved essentially tight bounds on the tower height of $n_p(\alpha,\beta)$ when $\alpha \leq 1/2$ and $p \geq 19$. We next discuss what happens when the set density $\alpha$ is close to one and prove Theorem \ref{thmverydensetight}, which gives a tight bound on the tower height in this regime. No bound on $p$ is needed. Some of this discussion works for all abelian groups of odd order, so we indicate where we are assuming the group is $\mathbb{F}_p^n$. 

Let $G$ be a finite abelian group of odd order and $f:G\rightarrow [0,1]$ have density $\alpha$. Since $\alpha$ is close to one, it will be convenient to work with the complementary function $g:G\rightarrow [0,1]$ given by $g(x)=1-f(x)$, which has density $\gamma := 1-\alpha$. The weight of a $3$-AP $(a,b,c)$ is $f(a)f(b)f(c)=(1-g(a))(1-g(b))(1-g(c)) \geq 1-g(a)-g(b)-g(c)$. As each element is in exactly three $3$-APs with a given common difference, for {\it every} nonzero $d$, the density of $3$-APs with common difference $d$ is at least $1-3\gamma$. This bound is best possible in $\mathbb{F}_3^n$ if $\gamma=1/3$ by considering the characteristic function of the subset which consists of all elements with $1$ or $2$ in its first coordinate and the common difference $d$ is nonzero in the first coordinate. However, the total density of $3$-APs is considerably larger than this bound. Indeed, the density of $3$-APs $(a,b,c)$ is 
\begin{eqnarray}\label{popdensecomp}
\mathbb{E}[f(a)f(b)f(c)] & = & \mathbb{E}[(1-g(a))(1-g(b))(1-g(c))] \nonumber \\ & = & 1-\mathbb{E}[g(a)+g(b)+g(c)]+\mathbb{E}[g(a)g(b)+g(a)g(c)+g(b)g(c)]-\mathbb{E}[g(a)g(b)g(c)] \nonumber \\ & =& 1-3\gamma+3\gamma^2-\mathbb{E}[g(a)g(b)g(c)] \nonumber \\ & \geq & 1-3\gamma+3\gamma^2-\mathbb{E}[g(a)g(b)]  = 
1-3\gamma+2\gamma^2 = \alpha^3-(\gamma^2-\gamma^3).
\end{eqnarray}

It is not difficult to see that this bound is best possible if $g$ is the indicator function for a subgroup of $G$. It follows from (\ref{popdensecomp}) by taking away the contribution from the trivial arithmetic progressions (those with $d=0$) that $n_p(\alpha,\beta) \leq n$ if 

\begin{equation}\label{popdensecomp1}
\beta \leq \alpha^3-\gamma^2+\gamma^3-\alpha/p^n <\left(\alpha^3-\gamma^2+\gamma^3-\alpha/p^n\right)\frac{p^n}{p^n-1}. \end{equation}

Recall that $\beta=\alpha^3-\epsilon$. The above discussion shows that, for $\epsilon \geq \gamma^2$, we have $n_p(\alpha,\beta) \leq 3\log_p (1/\gamma)$. We next discuss the proof of Theorem \ref{thmverydensetight}. 
We first prove the upper bound as given in the following theorem, which improves for $\alpha$ close to $1$ the tower height from the bound in Theorem \ref{thm:largebound}(1) by a factor $\Theta(\log(1/\gamma))$. Note that we may simply apply the bound in Theorem \ref{thm:largebound}(1) for $1/2 \leq \alpha < 59/60$ to get the range of $\alpha$ in Theorem \ref{thmverydensetight} not covered by the theorem below. 

\begin{thm}\label{uppverydensecase} Let  $\alpha \geq 59/60$, $\gamma=1-\alpha$ and $\epsilon=\alpha^3-\beta$. We have $$n_p(\alpha,\beta) \leq \tower(p,\log(\epsilon)/\log(36\gamma)+9,1/\epsilon).$$ 
\end{thm} 

The only difference between the proofs of Theorem \ref{uppverydensecase} and Theorem \ref{thm:largebound}(1) is that we repeatedly apply Lemma \ref{bettermeancube} below instead of Lemma \ref{lem:mean cubed density increment_small}.  Lemma \ref{bettermeancube} gives a better density increment at each step than Lemma \ref{lem:mean cubed density increment_small}, giving a factor $1+1/(36\gamma)$ instead of a factor two. 

\begin{lem}\label{bettermeancube}
Let $\alpha \geq 59/60$, $f:\mathbb{F}_p^n \rightarrow [0,1]$ satisfy $\mathbb{E}[f]=\alpha$, and $\gamma=1-\alpha \leq 1/60$. Let $H$ be a subspace of $\mathbb{F}_p^n$ with $|H|>\max(\gamma^{-3},4/\epsilon)$. If $\mathbb{E}_{y}[\lambda_{H+y}(f)] < \alpha^3-\epsilon$, then there is a subspace $H'$ of $H$ of codimension at most $\codim(H)+p^{\codim(H)}\cdot 144/\epsilon^2$ such that $b(H')-\alpha^3 \geq  \left(1+\frac{1}{36\gamma}\right)(b(H)-\alpha^3)+\frac{\epsilon}{20\gamma}$. 
\end{lem}
\begin{proof}
The proof begins along the lines of the proof of Lemma \ref{lem:mean cubed density increment_small} with $H=\mathbb{F}_p^n$. However, we can improve the bound on the mean cube density by a more careful analysis over simply using Schur's inequality. Denote the translates of $H$ by $H_j$, and denote the density of $f$ in $H_j$ by $\alpha_j=\alpha+\delta_j$, so $\mathbb{E}[\delta_j]=0$. Let $D$ be the number of translates of $H$. By the above discussion in (\ref{popdensecomp}) and (\ref{popdensecomp1}), the density of nontrivial $3$-APs in $H_j$ is at least $$\alpha_j^3-(1-\alpha_j)^2.$$ We will use this bound when $\delta_j < -5\gamma$. 

We apply the weak regularity lemma, Lemma \ref{lem:(Weak-regularity-lemma.)}, to each translate $H_j$ of $H$ with $\delta_j \geq -5\gamma$ to obtain an $\eta$-weakly-regular subspace $T_j$ of relative codimension at most $\lfloor\eta^{-2}\rfloor$ in $H$ with $\eta=\epsilon/12$. Denote the translates of $T_j$ in $H_j$ by $T_{jk}$ for $k \in H/T_j$. Let $D_j = |H/T_j| \le p^{\lfloor \eta^{-2}\rfloor}$. Denote the density of $f$ in $T_{jk}$ by $\alpha_{jk}=\alpha_j+\delta_{jk}$, so, for each $j$, we have $\mathbb{E}[\delta_{jk}]=0$. 
As in the proof of Lemma \ref{lem:mean cubed density increment_small}, the density of  $3$-APs in $H_j$ with nonzero common difference is at least $$\mathbb{E}[\alpha_{ja}\alpha_{jb}\alpha_{jc}]-6\eta=
\alpha_j^3+\mathbb{E}[\delta_{ja}\delta_{jb}\delta_{jc}]-6\eta,$$ 
where the expectations are over all triples of subspaces $T_{ja},T_{jb},T_{jc}$ that form a $3$-AP, and the equality is by expanding and using $\mathbb{E}[\delta_{ja}]=0$. 

Hence, 
\begin{align}
\mathbb{E}_{y}[\lambda_{H+y}(f)] &\ge \frac{1}{D}\left(\sum_{\delta_j<  -5\gamma}\left(\alpha_j^3-(1-\alpha_j)^2\right)+\sum_{\delta_j \geq  -5\gamma}  \left(\alpha_j^3+\mathbb{E}[\delta_{ja}\delta_{jb}\delta_{jc}]-6\eta\right)\right) \nonumber \\
&\ge b(H)-D^{-1}\sum_{\delta_j< -5\gamma}(1-\alpha_j)^2-6\eta+D^{-1}\sum_{\delta_j\geq  -5\gamma} \mathbb{E}[\delta_{ja}\delta_{jb}\delta_{jc}]. \label{largeden}
\end{align}

We next estimate the terms above. Observe that, for $ \delta_j \leq -5\gamma$, we have 
\begin{equation}\label{nextdense}(1-\alpha_j)^2=(\gamma-\delta_j)^2 \leq \left(\frac{6}{5}\delta_j\right)^2=\frac{36}{25}\delta_j^2. 
\end{equation}

Substituting in (\ref{nextdense}) and $\eta=\epsilon/12$ into (\ref{largeden}), 
\begin{align*}
\mathbb{E}_{y}[\lambda_{H+y}(f)] &\ge b(H)-D^{-1}\sum_{\delta_j<-5\gamma}\frac{36}{25}\delta_j^2-\epsilon/2+D^{-1}\sum_{\delta_j\geq  -5\gamma} \mathbb{E}[\delta_{ja}\delta_{jb}\delta_{jc}] \\
&\geq b(H)-\frac{36}{25}\mathbb{E}_j[\delta_j^2]-\epsilon/2+D^{-1}\sum_{\delta_j\geq  -5\gamma} \mathbb{E}[\delta_{ja}\delta_{jb}\delta_{jc}].
\end{align*}

We have 
$$b(H)-\alpha^3=\mathbb{E}_j[(\alpha+\delta_j)^3] -\alpha^3= 3\alpha \mathbb{E}_j[\delta_j^2]+\mathbb{E}_j[\delta_j^3] \geq  (3\alpha-1) \mathbb{E}_j[\delta_j^2] \geq  1.95 \, \mathbb{E}_j[\delta_j]^2.$$

Thus 
\begin{align}
\mathbb{E}_{y}[\lambda_{H+y}(f)]&\ge \alpha^3+\left(1-\frac{36/25}{1.95}\right)(b(H)-\alpha^3)-\epsilon/2+D^{-1}\sum_{\delta_j\geq  -5\gamma} \mathbb{E}[\delta_{ja}\delta_{jb}\delta_{jc}] \nonumber \\
&=\alpha^3+\frac{.51}{1.95}(b(H)-\alpha^3)-\epsilon/2+D^{-1}\sum_{\delta_j\geq  -5\gamma} \mathbb{E}[\delta_{ja}\delta_{jb}\delta_{jc}]. \label{nextdense2}
\end{align}

We still need to bound the last term. Observe that $\sum \delta_{ja}\delta_{jb}\delta_{jc} = D_j^2 \, \mathbb{E}[\delta_{ja}\delta_{jb}\delta_{jc}]$ as there are $D_j^2$ choices for the triple $(a,b,c)$ for which $T_{ja},T_{jb},T_{jc}$ are translates of $T_j$ in arithmetic progression, because there are $D_j$ choices for each of $a$ and $b$ (which uniquely determine $c$).  By throwing away the nonnegative terms, we have 

$$\sum \delta_{ja}\delta_{jb}\delta_{jc} \geq \sum_{\delta_{ja}\delta_{jb}\delta_{jc}<0} \delta_{ja}\delta_{jb}\delta_{jc}.$$

The above sum includes terms where all three of $\delta_{ja}$, $\delta_{jb}$, $\delta_{jc}$ are negative or exactly one is negative and the other two are positive. Hence, $$\sum_{\delta_{ja}\delta_{jb}\delta_{jc} < 0} |\delta_{ja}\delta_{jb}\delta_{jc}| \leq \sum_{\delta_{ja}< 0} |\delta_{ja}\delta_{jb}\delta_{jc}|+\sum_{\delta_{jb}< 0} |\delta_{ja}\delta_{jb}\delta_{jc}|+\sum_{\delta_{jc}< 0} |\delta_{ja}\delta_{jb}\delta_{jc}|,$$ as each term in the sums are nonnegative, and each term on the left hand side appears either once or three times on the right hand side. Thus, at least one of the three sums on the right hand side is at least $\frac{1}{3}\sum_{\delta_{ja}\delta_{jb}\delta_{jc} < 0} |\delta_{ja}\delta_{jb}\delta_{jc}| $. Without loss of generality, suppose it is the first sum, so $-\sum_{\delta_{ja}< 0} |\delta_{ja}\delta_{jb}\delta_{jc}| \leq \frac{1}{3} D_j^2\mathbb{E}[\delta_{ja}\delta_{jb}\delta_{jc}]$. By the arithmetic mean - geometric mean inequality, we have $|\delta_{ja} \delta_{jb} \delta_{jc} | \leq |\delta_{ja}| (\delta_{jb}^2+\delta_{jc}^2)/2$. Summing over all triples of subspaces in arithmetic progression, we obtain $$ \frac{1}{3} D_j^2\mathbb{E}[\delta_{ja}\delta_{jb}\delta_{jc}] \geq \sum_{\delta_{ja}< 0} \delta_{ja}\delta_{jb}^2 = \left(\sum_{\delta_{ja}<0} \delta_{ja} \right) \sum_{b} \delta_{jb}^2.$$  As $\delta_{ja}=\alpha_{ja}-\alpha_j \leq 1-\alpha_j \leq 1-\alpha +5\gamma = 6\gamma$ for all $j$ with $\delta_j \geq -5\gamma$, we have  $\sum_{\delta_{ja} \geq 0} \delta_{ja} \leq 6\gamma D_j$. As also $\mathbb{E}_a[\delta_{ja}]=0$, then 
$\sum_{\delta_{ja}<0} \delta_{ja} \geq -6\gamma D$. Hence,  $$\frac{1}{3} D_j^2\mathbb{E}[\delta_{ja}\delta_{jb}\delta_{jc}] \geq (-6\gamma D_j)\sum_{b}\delta_{jb}^2,$$
and so $\mathbb{E} [\delta_{ja}\delta_{jb}\delta_{jc}] \geq -18\gamma \mathbb{E}_b[\delta_{jb}^2]$. 

Substituting into (\ref{nextdense2}), 

\begin{equation}\label{nextedense3} \mathbb{E}_{y}[\lambda_{H+y}(f)]\ge \alpha^3+\frac{.51}{1.95}(b(H)-\alpha^3)-\epsilon/2-D^{-1}\sum_{\delta_j \geq -5\gamma}18\gamma \mathbb{E}_b[\delta_{jb}^2].
\end{equation}

As $\mathbb{E}_{y}[\lambda_{H+y}(f)] < \alpha^3-\epsilon$, we therefore obtain

\begin{equation}\label{usefulbound6}18\gamma D^{-1}\sum_{\delta_j \geq -5\gamma} \mathbb{E}_b[\delta_{jb}^2] \geq \frac{.51}{1.95}(b(H)-\alpha^3)+\epsilon/2.
\end{equation}

Let $H'= \bigcap_{j} T_j$. The codimension of $H'$ is at most $\codim(H)+\eta^{-2}p^{\codim(H)}=144\epsilon^{-2}p^{\codim(H)}+\codim(H)$. By convexity, we have  
$$b(H') \geq D^{-1}\left(\sum_{\delta_j<-5\gamma} \alpha_j^3+\sum_{\delta_j \geq -5\gamma}\mathbb{E}_a[\alpha_{ja}^3]\right).$$
By expanding with $\alpha_{ja}=\alpha_j+\delta_{ja}$ and using $\mathbb{E}[\delta_{ja}]=0$, we have $$\mathbb{E}_a[\alpha_{ja}^3]=\alpha_j^3+3\alpha\mathbb{E}_a[\delta_{ja}^2]+\mathbb{E}_a[\delta_{ja}^3] \geq 
\alpha_j^3+(3\alpha-1)\mathbb{E}_a[\delta_{ja}^2] \geq \alpha_j^3+1.95 \, \mathbb{E}_a[\delta_{ja}^2].$$ 
Thus,
$$b(H') \geq \mathbb{E}[\alpha_j^3]+1.95D^{-1}\sum_{\delta_j \geq -5\gamma} \mathbb{E}_a [\delta_{ja}^2]=b(H)+1.95D^{-1}\sum_{\delta_j \geq -5\gamma} \mathbb{E}_a [\delta_{ja}^2].$$ 
and so $1.95^{-1}\left(b(H')-b(H)\right) \geq D^{-1}\sum_{\delta_j \geq -5\gamma} \mathbb{E}_a [\delta_{ja}^2]$. Substituting into (\ref{usefulbound6}), we obtain 
$$18\gamma \cdot \frac{1}{1.95}\left(b(H')-b(H) \right)\geq \frac{.51}{1.95}(b(H)-\alpha^3)+\epsilon/2.$$
Substituting in $b(H')-b(H)=b(H')-\alpha^3-(b(H)-\alpha^3)$ and rearranging, we have $$b(H')-\alpha^3 \geq \left(1+\frac{.51}{18\gamma}\right)(b(H)-\alpha^3)+\frac{1.95}{36\gamma}\epsilon \geq \left(1+\frac{1}{36\gamma}\right)(b(H)-\alpha^3)+\frac{\epsilon}{20\gamma},$$
completing the proof. 
\end{proof}

We next discuss the proof of the lower bound in Theorem \ref{thmverydensetight}. Note that there must be a reason why Theorem 8 in \cite{FPI} (the weighted analogue which is used to deduce Theorem \ref{thm:Lower bound for line density}) does not hold for $\alpha$ close to one, as the lower bound it claims would be better than the upper bound we just got. We next discuss where in the proof of Theorem 8 in \cite{FPI} we relied on $\alpha$ not being close to one, and discuss  how to properly modify it to get the lower bound in Theorem \ref{thmverydensetight}. Since the proof is only a minor modification of Theorem 8 in \cite{FPI}, for brevity we do not repeat the details and only specify the key difference.
 
We reuse the notations in \cite{FPI}. The first level of the construction is identical. Observe that in the proof of Theorem 8 in \cite{FPI}, at each level $i \geq 2$ for $x \in H_i$ we partition the subspace with the first $n_{i-1}$ coordinates equal to $x$ into $p$ affine subspaces of relative codimension one (which are translates of each other) depending on the dot product of the last $m_i$ coordinates with $v(x)$, and for some of these codimension one subspaces we make the value of $f_i$ equal to zero, and on the other subspaces we make the value $\frac{1}{\zeta}(1+\eta)\alpha \approx \frac{3}{2}\alpha$. The problem with this construction for $\alpha$ close to one is that we would get a density which is bigger than one on some subspaces, and this is not allowed. So instead of making the values $0$ or something else (namely $\frac{1}{\zeta}(1+\eta)\alpha$) to keep the average value to be $(1+\eta)\alpha$ on these subspaces, we make the  values $1$ (instead of $\frac{1}{\zeta}(1+\eta)\alpha$) on the denser subspaces and $(1-\zeta)^{-1}\left((1+\eta)\alpha-\zeta\right)$ (instead of $0$) on the sparser subspaces, to keep the average value to be $(1+\eta)\alpha$ and all values in $[0,1]$. The rest of the proof is the same, apart from appropriately modifying the parameters. In particular, the exponential growth of $\mu_i$ (which is the fraction of the space $\mathbb{F}_p^{n_{i-1}}$ that $H_i$ takes up) has the base exponential constant in this version $\gamma^{-O(1)}$ instead of $90$ in order to counteract the smaller decrease in $3$-AP density that we get from the modification described above.

\section{Concluding remarks}

\noindent \textbf{Arithmetic Progressions in Groups}
\vspace{0.1cm}

Green \cite{Green05} further proved that Theorem \ref{thm:Original theorem}
holds not just in $\mathbb{F}_{p}^{n}$ but in any abelian group $G$ of odd order or in $[N]=\{1,2,\ldots,N\}$. In joint work with Zhao \cite{FPZ}, we extend Theorem 2 of \cite{FPI} to the interval setting. We also generalize the upper bound in Theorem 2 of \cite{FPI} to general abelian groups of odd order. A substantial difficulty in this setting is dealing with the possible lack of subgroups. The upper bound proof uses Fourier analysis on Bohr sets to extend the ideas here. With some additional ideas, the lower bound can be generalized to cyclic groups which can be written as a product of groups with appropriate growth in size. The construction however runs into a serious obstruction in the case the group has few subgroups, for example, when the group is a primal cyclic group. 

A major direction in additive combinatorics in recent years has been
to extend results from abelian groups to all (not necessarily abelian) groups.
For example, the Freiman-Ruzsa theorem gives a characterization of subsets in abelian groups that have small doubling, i.e., sets $A$ for which $|A+A|=O(|A|)$. Breuillard, Green,
and Tao \cite{BGT} have recently extended this to nonabelian groups,
which has diverse applications. 

Another important example
is Roth's theorem, which was extended by Bergelson, McCutcheon, and Zhang \cite{BMZ} to the nonabelian setting.  
This result says that if $G$ is a group of order $N$ and $A\subset G$ is without a nontrivial solution to $xy=z^{2}$ with $x,y,z$
distinct\footnote{Here, of course, the group operation is written multiplicatively.},
then $|A|=o(N)$. This also follows from the arithmetic triangle removal lemma of Kr\'al', Serra, and Vena \cite{KSV}
(see also \cite{Soly}), which was proved through an application of the triangle removal lemma. However, by using the triangle removal lemma, the proof gives a weak quantitative estimate. Recently, Sanders \cite{Sanders16} used representation theory in order to extend the standard Fourier proof to the nonabelian setting and give a new proof of this nonabelian Roth's theorem with the bound $|A|=N/(\log \log N)^{\Omega(1)}$. 

Another natural extension of Roth's theorem to groups states that if $G$ is a group of order $N$, and $A \subset G$ has no triple $x,xd,xd^2$ with $d \not = 1$, then $|A|=o(N)$. Pyber \cite{Pyber} proved that every group $G$ of order $N$ has an abelian subgroup $H$ of order at least $2^{\Omega(\sqrt{\log N})}$, which is in general best possible. By applying Roth's theorem for abelian groups in the cosets of $H$, we get this Roth-type theorem in the group $G$, and with a reasonable quantitative bound, see Solymosi \cite{Soly}.

We think it would be interesting to know if Green's theorem, Theorem \ref{thm:Original theorem}, extends to nonabelian groups. One version asks: does there exist, for each $\epsilon>0$,  an $N_1(\epsilon)$ such that for every group $G$ of order at least $N_1(\epsilon)$ and every subset $A \subset G$ of density $\alpha$, there is nonidentity element $d \in G$ such that the density of triples $x,xd,xd^2$ which are in $A$ is at least $\alpha^3-\epsilon$. If $G$ has an abelian subgroup $H$ of index $O(1)$, then applying the regularity-type upper bound argument, starting with the subgroup $H$, shows that such a result holds in this case with a similar bound on $N_1(\epsilon)$ as in the abelian case. 

A variant of this question asks: does there exist, for each $\epsilon>0$, an $N_2(\epsilon)$ such that if $G$ is a group of order at least $N_2(\epsilon)$ and $A \subset G$ has density $\alpha$, then there is a nonidentity element $d \in G$ such that the density of triples $x,y,z$ with $xz=y^2$ and $yx^{-1}=d$ which are in $A$ is at least $\alpha^{3}-\epsilon$? For quasirandom groups, which were introduced by Gowers \cite{Gow08}, it is easy to show that this version holds as the density of solutions to $xz=y^{2}$ which are in $A$ is asymptotically $\alpha^{3}-o(1)$. Further, it holds for quasirandom groups with only a polynomial bound on $N_2(\epsilon)$, instead of the tower-type bound as in the abelian case. This shows that, unlike in the abelian case, the quantitative bounds can depend substantially on the group structure.   

One possible approach to proving a nonabelian Green's theorem is by developing a nonabelian generalization of the arithmetic regularity lemma, which would likely have further applications. One would likely want to develop a nonabelian version of Bohr sets. Maybe some of the ideas on approximate subgroups in the important work of Breuillard, Green, and Tao \cite{BGT} or from representation theory as in the work of Gowers on quasirandom groups \cite{Gow08}, \cite{Gow17} could be helpful here. 

\vspace{0.2cm}
\noindent \textbf{Four-term arithmetic progressions}
\vspace{0.1cm} 

\noindent{Green and Tao \cite{Green05a,Green07,GrTa10}
proved that for each $\epsilon>0$ there is $n'(\epsilon)$ such that
for any $n\geq n'(\epsilon)$ and any subset of $\mathbb{F}_{5}^{n}$
of density $\alpha$, there is a nonzero $d\in\mathbb{F}_{5}^{n}$
such that the density of four-term arithmetic progressions with common
difference $d$ is at least $\alpha^{4}-\epsilon$. Ruzsa \cite{BHKR} proved that an analogous result does not hold for longer lengths. 

We think it would be interesting
to estimate $n'(\epsilon)$. Does it grow as a tower-type function? It appears the proof given in \cite{Green07} with the bound on the inverse Gowers theorem for $U^3$ from \cite{GrTa08} gives a wowzer-type upper bound, which is in the next level of the Grzegorczyk hierarchy after the tower function. The lower bound construction we presented here for three-term arithmetic progressions can be modified to give a lower bound on $n'(\epsilon)$ which is a tower of twos of height $\Theta(\log(1/\epsilon))$. To get an improved upper bound, some of the ideas of Green and Tao in the paper \cite{GrTa12} might be useful.

\vspace{0.2cm}
\noindent \textbf{Multidimensional generalization of cap sets and popular differences}
\vspace{0.1cm} 

Recall that the multidimensional cap set problem discussed in the beginning of the introduction asks to estimate the maximum size $r(n,m)$ of a subset of $\mathbb{F}_3^n$ which does not contain a $m$-dimensional affine subspace, and $N^{1-(m+1)3^{-m}} \leq r(n,m) \leq (1+o(1))N^{1-C^m}$, where $ C \approx 13.901$ is an explicit constant. 

It remains an interesting problem to tighten the bounds on $r(n,m)$.
Is the right exponential constant $3$, which comes from the random bound, or is it $C$, which comes from
applying the arithmetic triangle removal lemma, or is it something in between?

We have the following multidimensional generalization of Green's theorem. 

\begin{thm}\label{multiGreentheorem}
For every $\epsilon>0$ and positive integer $r$, there is a (least) positive integer $n(r,\epsilon)$ such that for every $n \geq n(r,\epsilon)$ and $A \subset \mathbb{F}_3^n$ of density $\alpha$, there is a subspace $S$ of dimension $r$ such that the density of translates of $S$ in $A$ is at least $\alpha^{3^r}-\epsilon$.   
\end{thm}

The multidimensional cap set result can be generalized to vector spaces over a fixed finite field, and Theorem \ref{multiGreentheorem} to abelian groups of odd order or to intervals, but we need to replace the notion of subspace by ``box'', also known as a generalized arithmetic progression. A $k$-box $B$ of dimension $r$ is a set of the form $\{a_0+i_1d_1+i_2d_2+\cdots+i_rd_r: 0 \leq i_j \leq k-1~\textrm{for}~1 \leq j \leq r\}$. So a $k$-box of dimension one is just a $k$-term arithmetic progression. It is proper if all the elements are distinct, that is, if $|B|=k^r$. We refer to $d_1,\ldots,d_r$ as the common differences of the $k$-box, and $a_0$ as the initial term. 

\begin{thm}\label{multiGreentheoremgeneral}
For every $\epsilon>0$ and positive integer $r$, there is a (least) positive integer $N(r,\epsilon)$ such that the following holds. For every $N \geq N(r,\epsilon)$, if $G$ is an abelian group of odd order $N$ or $G=[N]$, and $A \subset G$ of density $\alpha$, then there are $d_1,\ldots,d_r$ such that the $3$-boxes of dimension $r$ with common differences $d_1,\ldots,d_r$ are proper and the number of them in $A$ is at least $(\alpha^{3^r}-\epsilon)N$. 
\end{thm}

The proof of Theorem \ref{multiGreentheoremgeneral} is by induction on $r$. The base case $r=1$ is simply Green's theorem. Suppose we would like to prove it for $r>1$. Let $A \subset G$ with $|G|=N$ sufficiently large and $|A|=\alpha N$. We apply the induction hypothesis to $A$ with parameters $r-1$ and  $\epsilon/4$ to obtain $d_1,\ldots,d_{r-1}$ such that the sums $i_1d_1+\cdots+i_{r-1}d_{r-1}$ with $0 \leq i_j \leq 2$ for $1 \leq j \leq r-1$ are distinct, the number of $a_0 \in [N]$ for which $\{a_0+i_1d_1+i_2d_2+\cdots+i_{r-1}d_{r-1}: 0 \leq i_j \leq k-1~\textrm{for}~1 \leq j \leq r-1\} \subset A$ is at least $\left(\alpha^{3^{r-1}}-\epsilon/4\right)N$, and we let $A_0$ be the set of such initial terms $a_0$. The proof of Green's theorem shows that not only is there a single nonzero $d$ which is a popular difference, but in fact a positive constant fraction (depending on $\epsilon$) of $d$ are popular differences. Applying Green's theorem to $A_0$, we get many choices of $d_r$ (more than $5^r$ suffices) such that the number of $3$-APs in $A_0$ of common difference $d_r$ is at least 
$\left(\left(\alpha^{3^{r-1}}-\epsilon/4\right)^3-\epsilon/4\right)N \geq \left(\alpha^{3^r}-\epsilon\right)N$. We can then choose such a $d_r$ such that all the sums  $i_1d_1+\cdots+i_{r}d_{r}$ with $0 \leq i_j \leq 2$ for $1 \leq j \leq r$ are distinct. The $r$-dimensional $3$-boxes with common differences $d_1,\ldots,d_r$ are proper and at least $\left(\alpha^{3^r}-\epsilon\right)N$ of them are contained in $A$, completing the proof. This proof shows that for $r$ fixed, $N(r,\epsilon)$ in Theorem \ref{multiGreentheoremgeneral} grows at most as a tower of twos of height $\Theta_r(\log(1/\epsilon))$. A modification to our lower bound constructions shows that this bound is tight over $\mathbb{F}_p^{n}$ for each fixed $r$. For brevity, we leave out the details of the proof.  

The proof of Theorem \ref{multiGreentheoremgeneral} together with the result of Green and Tao \cite{GrTa10} for four-term progressions shows that Theorem \ref{multiGreentheoremgeneral} also holds with the length $3$ replaced by $4$. 

\vspace{0.2cm}
\noindent \textbf{From Tower to Wowzer}
\vspace{0.1cm} 

Ron Graham \cite{Graham} asked if faster growing functions like wowzer-type (the next level in the Grzegorczyk hierarchy after tower) naturally appear in similar problems. Formally, the wowzer function is defined by $\textrm{Wow}(1)=2$ and $\textrm{Wow}(n)=\tower(\textrm{Wow}(n-1))$, where $\tower(n)=\tower(2,n)$ is an exponential tower of twos of height $n$.  

We first remark that Green's proof of his theorem, which is obtained by directly applying the arithmetic regularity lemma and the counting lemma, gives the following strengthening. It is stronger by the fact that, for a set of density $\alpha$, the mean cube density of a subspace  is at least $\alpha^3$ by convexity.

\begin{thm}\label{regularitymeancube}
For each $\epsilon>0$  and $p$, there is a least positive integer $m_p(\epsilon)$ such that for every $A \subset \mathbb{F}_p^n$ of density $\alpha$, there is a subspace $H$ of $\mathbb{F}_p^n$ of codimension at most $m_p(\epsilon)$ such that the density of $3$-APs with common difference in $H$ is at least $b(H)-\epsilon$. 
\end{thm}

The original proof of Theorem \ref{regularitymeancube} using the arithmetic regularity lemma gives an upper bound on $m_p(\epsilon)$ for $p$ fixed which grows as a tower of height $\epsilon^{-O(1)}$. Adapting the upper bound proof from \cite{FPI} gives a better upper bound which is a tower of height $\Theta(1/\epsilon)$. We can get a matching lower bound, so a tower of height $\Omega(1/\epsilon)$, by modifying the construction used to give a lower bound in Green's theorem from \cite{FPI}. In the lower bound construction in \cite{FPI}, the fraction of subspaces we make perturbations to increases by a constant factor at each step. To get the lower bound here, after the first level, the fraction of subspaces we make perturbations to at each step is the same.  For brevity, we leave out the details of the proof.  

\begin{thm}\label{regularitymeancube1}
The function $m_p(\epsilon)$ defined in Theorem \ref{regularitymeancube} for $p$ fixed grows as a tower of height $\Theta(1/\epsilon)$. 
\end{thm}

While the tower height grows faster in the above result than in Green's theorem, it is still in the same level of the Grzegorczyk hierarchy. To go to the next level of the Grzegorczyk hierarchy, it is natural to try to find an extremal problem that essentially encodes an arithmetic strong regularity lemma. The (graph) strong regularity lemma of Alon, Fischer, Krivelevich, and Szegedy \cite{AFKS} finds a pair of partitions $P$ and $Q$, with $Q$ a refinement of $P$, and the regularity of $Q$ is allowed to depend on the size of $P$. The proof of the strong regularity lemma involves applying Szemer\'edi's regularity lemma at each step, and so the bound one gets on the number of parts of $Q$ is iteratively applying the tower function $\epsilon^{-\Theta(1)}$ times. In other words, it is of wowzer-type. That such a bound is necessary was shown by Conlon and the first author \cite{CF12} and independently with a weaker wowzer-type by Kalyanasundaram and Shapira \cite{KaSh}. 

An arithmetic analogue of the strong regularity lemma is as follows. For each function $g:\mathbb{Z}_{\geq 0} \rightarrow (0,1)$, there is $M_p(g)$ such that the following holds. If $A \subset \mathbb{F}_p^n$, then there are subspaces $H_1 \subset H_0 \subset \mathbb{F}_p^n$ such that the codimension of $H_1$ is at most $M_p(g)$, $b(H_1) \leq b(H_0)+g(0)$, and $H_1$ is $g(m)$-regular, where $m$ is the codimension of $H_0$. As long as $1/g(n)$ grows not too slowly and not too fast with $n$, which here means it is bounded between any constant number of iterations of the logarithmic function and any constant number of iterations of the exponential function, then the function $M_p(g)$ grows as wowzer in $\Theta(1/\epsilon)$ where $\epsilon=g(0)$. 

The next result follows easily from the arithmetic strong regularity lemma and the counting lemma. 

\begin{thm}\label{wowzer1}
For each function $g:\mathbb{Z}_{\geq 0} \rightarrow (0,1)$, there is $M'_p(g)$ such that the following holds. If $A \subset \mathbb{F}_p^n$ then there are subspaces $H_1 \subset H_0 \subset \mathbb{F}_p^n$ such that the codimension of $H_1$ is at most $M'_p(g)$ and the density of $3$-APs with common difference in $H_1$ is in the interval $[b(H_0)-g(m),b(H_0)+g(0)]$, where $m$ is the codimension of $H_0$. 
\end{thm}

\begin{thm}
The function $M'_p(g)$ in Theorem \ref{wowzer1} with $g(n)=\epsilon/(n+1)$ grows as wowzer in $\Theta(1/\epsilon)$.
\end{thm}

The upper bound follows directly from the proof of the arithmetic strong regularity lemma. The lower bound is by appropriately modifying the lower bound construction for Green's theorem with modifications similar to that used in \cite{CF12} to mimic the upper bound proof of the strong regularity lemma.  For brevity, we leave out the details of the proof.  

\vspace{0.2cm}
\noindent \textbf{Monochromatic arithmetic progressions with popular differences}
\vspace{0.1cm}

Van der Waerden's theorem \cite{vdw} states that, for all positive integers $k$ and $r$, there exists $W(k,r)$ such that if $N \geq W(k,r)$, then every $r$-coloring of $\mathbb{Z}_N$ contains a monochromatic $k$-term arithmetic progression. Many results in Ramsey theory are of this flavor, that in any finite coloring of a large enough system, there is a monochromatic pattern. In some instances, a stronger density-type theorem also holds, showing that any dense set contains the desired pattern. This is the case for van der Waerden's theorem, as Szemer\'edi's theorem \cite{Sz75} is such a strengthening, implying that the densest of the $r$ color classes will necessarily contain the desired arithmetic progression. Szemer\'edi's theorem states that for each positive integer $k$ and $\epsilon>0$, there is $S(k,\epsilon)$ such that, if $N \geq S(k,\epsilon)$, then any subset of $\mathbb{Z}_N$ of size at least $\epsilon N$ contains a $k$-term arithmetic progression. Note that Roth's theorem is the case $k=3$. By a Varnavides-type averaging argument, one can further show that a stronger, multiplicity version of van der Waerden's theorem (and of Szemer\'edi's theorem) holds, which shows that a fraction $c(k,r)-o(1)$ of the $k$-term arithmetic progressions must be monochromatic. Observe that a random coloring gives an upper bound on $c(k,r)$ of $r^{1-k}$. For $r>2$ it is possible to show that there are colorings with relatively few monochromatic arithmetic progressions, considerably smaller than the random bound. For example, using the Behrend construction giving a lower bound for Roth's theorem, one can construct an $r$-coloring of $\mathbb{Z}_n$ such that the fraction of three-term arithmetic progressions which are monochromatic is only $r^{-\Omega(\log r)}$, which is much less than the random bound of $r^{-2}$. 

Let $G$ be an abelian group of odd order. Note that, in a random coloring of $G$, for each nonzero $d$, the density of $3$-APs with common difference $d$ will likely be concentrated around $\frac{1}{r^2}$. Just like in the density version, we can get arbitrarily close to the random bound for the most popular difference. Green \cite{Green05} proved that, for $r$ fixed, the arithmetic regularity lemma in $G$ extends to $r$ subsets of $G$ (in particular, for the $r$ color classes in an $r$-coloring of $G$), so that the decomposition is regular with respect to each of the $r$ subsets. Green's proof of the density theorem on arithmetic progressions with popular differences extends in a straightforward manner to obtain the following coloring variant. Indeed, using this extension of the arithmetic regularity lemma and scaling the approximation parameter $\epsilon$ by $r$, as long as $|G| > N(\epsilon,r)$, there is a nonzero $d$ such that, for each color $i$, the density of $3$-APs with common difference $d$ which are monochromatic with color $i$ is at least $\alpha_i^3 - \frac{\epsilon}{r}$, where $\alpha_i$ is the density of color $i$. Summing over all $i$, we get the density of monochromatic $3$-APs with common difference $d$ is at least $\sum_{i=1}^r (\alpha_i^3 - \epsilon/r) \geq r(1/r)^3 -\epsilon = \frac{1}{r^2} -\epsilon$, where the inequality uses Jensen's inequality applied to the convex function $f(x)=x^3$ and the average of the $\alpha_i$ is $1/r$. 

\begin{thm}
For each $\epsilon>0$ and positive integer $r$, there is $N=N(\epsilon,r)$ such that if $G$ is an abelian group with odd order $|G| \geq N$, then for every $r$-coloring of $G$, there is a nonzero $d \in G$ such that the density of $3$-APs with common difference $d$ that are monochromatic is at least $\frac{1}{r^2}-\epsilon$.   
\end{thm}

Picking the most popular of the $r$ colors, we have the following corollary. 

\begin{cor}
For each $\epsilon>0$ and positive integer $r$, there is $N'=N'(\epsilon,r)$ such that if $G$ is an abelian group of odd order with $|G| \geq N'$, then for every $r$-coloring of $G$, there is a nonzero $d$ and a color $i$ such that the density of $3$-APs with common difference $d$ that are monochromatic in color $i$ is at least $\frac{1}{r^3}-\epsilon$.   
\end{cor}

For $r=2$, these results are actually quite simple to prove and with bounds that are much better than applying the arithmetic regularity lemma because of the folklore observation that the total number of monochromatic $3$-APs is determined by the size of the first color class. Indeed, if $R$ and $B$ are the two color classes, so $|B|=|G|-|R|$, then we can count the number $P$ of three-term arithmetic progressions with a distinguished pair of elements of the same color in two different ways. First, for each pair of elements there are three arithmetic progressions containing that pair, so we get $P=3{|R| \choose 2}+3{|B| \choose 2}$. Alternatively, every three-term arithmetic progression contains one or three such monochromatic pairs, and there are ${|G| \choose 2}$ such three-term arithmetic progressions in $G$, so we also get $P={|G| \choose 2}+2M$, where $M$ is the number of monochromatic three-term arithmetic progressions. We thus get $M=\frac{1}{2}\left(3{|R| \choose 2}+3{|B| \choose 2}-{|G| \choose 2}\right)=\frac{|G|^2}{2}-\frac{3}{2}|R|(|G|-|R|)-\frac{|G|}{2}$. This is maximized when $|R|=\left\lfloor \frac{|G|}{2} \right\rfloor=(|G|-1)/2$, and we get $(|G|^2-4|G|+3)/8$ monochromatic arithmetic progressions in this case. Thus, the density of monochromatic $3$-APs is at least $\left((|G|^2-4|G|+3)/8\right)/{|G| \choose 2}=\frac{1}{4}-\frac{3}{4|G|}$. Hence, as long as $|G| \geq \frac{3}{4}\epsilon^{-1}$, then the density of monochromatic $3$-term APs at least $\frac{1}{4}-\epsilon$, and it follows that $N(\epsilon,2) \leq \frac{3}{4}\epsilon^{-1}$. Thus we get a linear upper bound in $1/\epsilon$ on $N(2,\epsilon)$, much smaller than the tower-type bound that comes from applying the arithmetic regularity lemma.

However, for $r \geq 3$, there is no such simple formula for the number of monochromatic $3$-APs in a $r$-coloring. In fact, we can prove a coloring variant of Theorem 2 in \cite{FPI}, showing that $N(\epsilon,r)$ and $N'(\epsilon,r)$ for $r \geq 3$ grow as a tower of twos of height $\Theta_r(\log (1/\epsilon))$. 

\begin{thm}
For $r \geq 3$ fixed, we have $N(\epsilon,r)$ and $N'(\epsilon,r)$ grow as a tower of twos of height $\Theta_r(\log(1/\epsilon))$. 
\end{thm}

Further discussion and proofs are contained in \cite{FP}.

\vspace{0.2cm}
\noindent \textbf{Linear equations}
\vspace{0.1cm} 

A well known conjecture of Sidorenko \cite{Si93} and Erd\H{o}s-Simonovits \cite{Simon} states that if $H$ is a bipartite graph, then the random graph with edge density $\alpha$ has in expectation asymptotically the minimum density of copies of $H$ (which is $\alpha^{e(H)}$) over all graphs with the same number of vertices and edge density. Simple constructions show that the assumption that $H$ is bipartite is necessary. A stronger conjecture, known as the forcing conjecture, states that if $H$ is bipartite and contains a cycle, and $G$ is a graph with edge density $\alpha$ and $H$-density $\alpha^{e(H)}+o(1)$, then $G$ is quasirandom with edge density $\alpha$. Sidorenko's conjecture and the stronger forcing conjecture are still open but are now known to be true for a large class of bipartite graphs, see \cite{CoFoSu,CKLL,CoLee,Ha,KLL,LiSz,Sz}.  

Saad and Wolf \cite{SaWo} began the systematic study of analogous questions for linear systems of equations in finite abelian groups. While much of this discussion extends to general finite abelian groups,  for simplicity we restrict our attention to $\mathbb{F}_p^n$. Let $L(x_1,...,x_k)=\sum_{i=1}^{k} a_ix_i$ be a linear form with coefficients $a_i \in \mathbb{F}_p \backslash \{0\}$. The linear homogeneous equation $L=0$ is called {\it Sidorenko} if for every subset $A \subset \mathbb{F}_p^n$ of density $\alpha$, the density of solutions to $L=0$ which are in $A$ is at least $\alpha^k$, the random bound. We say that the equation $L=0$ is {\it matched} if the coefficients of $L$ can be paired up so that each pair sums to zero. It is a simple application of the Cauchy-Schwarz inequality that if $L=0$ is matched, then it is Sidorenko. Zhao and the authors \cite{FPZ2} recently proved that $L=0$ is Sidorenko if and only if it is matched. 

The equation $L=0$ is called {\it common} if for every $2$-coloring of $\mathbb{F}_p^n$, the density of monochromatic solutions to $L=0$ is at least $2^{1-k}$, the random bound.  It is easy to see that if $L=0$ is Sidorenko, then it is common. Cameron et al. \cite{Cameron} observed that if $k$ is odd, then $L=0$ is common. Saad and Wolf \cite{SaWo} conjectured that if $k$ is even, then $L=0$ is common if and only if it is matched. Zhao and the authors \cite{FPZ2} also proved this conjecture.  

The popular difference property can be generalized to general linear equations. We say that $L=0$ is {\it popular} if, for each $\epsilon>0$, if $n \geq n_L(\epsilon)$ and $A \subset \mathbb{F}_p^n$ has density $\alpha$, then there are nonzero and distinct $d_1,\ldots,d_{k-1}$ such that the density of solutions to $L$ with $x_{i+1}-x_1=d_i$ for $i=1,\ldots,k-1$ is at least $\alpha^k - \epsilon$. In particular, when $k=3$, and $a_1=a_2=1$ and $a_3=-2$ (viewed as elements of $\mathbb{F}_p$), then $n_L(\epsilon)=\max_{\alpha} n_p(\alpha,\alpha^3-\epsilon)$. Note that if $L=0$ is Sidorenko, then simply by averaging, $L=0$ is popular and furthermore, $n_L(\epsilon)$ is bounded above by $O(\log(1/\epsilon))$. We say that the linear homogeneous equation $L=0$ is translation invariant if and only if the sum of the coefficients is zero. Theorem \ref{thm:Original theorem} can be extended to show that $L=0$ is popular if and only if $L=0$ is translation invariant. Indeed, if $L=0$ is not translation invariant, then the affine subspace $S$ of codimension one (so density $\alpha=1/p$) consisting of those elements whose first coordinate is $1$ has no solution to $L=0$. Hence, it follows that $n_{L}(\epsilon)$ does not exist for $\epsilon < 1/p^k$. On the other hand, if $L=0$ is translation invariant, then it follows from the arithmetic regularity lemma proof that there is a regular subspace $H$, and the counting lemma and Jensen's inequality gives that the density of $L$ with $x_1,\ldots,x_k$ all in the same translate of $H$ is at least almost $\alpha^k$. By throwing out the solutions with $x_{i}=x_{j}$ for some $i\ne j$ (which is of smaller order) and averaging, we get that there exists nonzero $d_1,\ldots,d_{k-1} \in H$ for which the density of solutions to $L$ with $x_{i+1}-x_1=d_i$ for $1 \leq i \leq k-1$ is at least $\alpha^k-\epsilon$. This proof gives an upper bound on $n_L(\epsilon)$ which is a tower of height $\epsilon^{-O(1)}$. In fact we can directly adapt the proof of Theorem 2 in \cite{FPI} to show that $n_L(\epsilon)$ is bounded above by a tower of height $\Theta(\log(1/\epsilon))$, using a density increment argument with the mean $k$-th power density, defined as $b_k(H)=\mathbb{E}_{x}[f_H(x)^k]$. 

We note that our construction of the lower bound in Theorem 2 of \cite{FPI} and the lower bound in Theorem \ref{main2} heavily depends on the fact that the equation $x_1-2x_2+x_3=0$ is not Sidorenko, as a crucial ingredient of our construction is a model function with low $3$-AP density. As mentioned above, if $L=0$ is Sidorenko, then the lower bound on $n_L(\epsilon)$ is not of tower type, but in fact, only logarithmic in $\epsilon$. However, the tower-type lower bound in Theorem 2 of \cite{FPI} does not necessarily hold for a general linear equation which is not Sidorenko. In fact, we can construct explicit examples of linear homogeneous equations in $2^t$ variables for $t\ge 3$ which are not Sidorenko but $n_L(\epsilon)=O(\log(\epsilon^{-1}))$. On the other hand, when the number of variables in the equation $L=0$ is $3$ as in the case of $3$-APs, we can adapt directly the proof of Theorem 2 of \cite{FPI} to show that if an equation $L(x_1,x_2,x_3)=0$ is not Sidorenko, then $n_L(\epsilon)$ grows as a tower of height proportional to $\log(\epsilon^{-1})$. 

The above discrepancy suggest that it could be very interesting to understand and characterize the growth rate of $n_L(\epsilon)$ for linear equations in at least four variables. In fact, we do not know if there are any linear equations in at least four variables for which $n_L(\epsilon)$ grows as a tower function. 

We refer the reader to \cite{FPZ2} for further details and questions.

\vspace{0.2cm}
\noindent \textbf{Beyond the random bound when there are relatively few total arithmetic progressions}
\vspace{0.1cm} 

We can strengthen Green's theorem, Theorem \ref{thm:Original theorem}, as follows. If our space is large enough and the total density of three-term arithmetic progressions is substantially less than the random bound, then there is a nonzero $d$ for which the density of three-term arithmetic progressions with common difference $d$ is substantially larger than the random bound. 

\begin{thm}
For each $\alpha$ and $\beta<\alpha^3$, there is $\gamma>\alpha^3$ such that the following holds. For each $\delta>0$ there is $n'_p(\delta)$ such that if $n>n'_p(\delta)$ and $f:\mathbb{F}_p^n\rightarrow [0,1]$ has $3$-AP density at most $\beta$, then there is a nonzero $d$ such that the density of $3$-APs with common difference $d$ is at least $\gamma-\delta$. In particular, we can always take $\gamma=\alpha^3+(\alpha^3-\beta)/2$, and if $\alpha^3>2^{8+8C_p}\beta$, then we can take $\gamma=(\alpha^3/\beta)^{1/(2C_p)}\alpha^3$, where $C_p=\Theta(\log p)$ is the exponential constant in the arithmetic removal lemma. 
\end{thm}

The proof is as follows. Let $H_0=\mathbb{F}_p^n$, so $b(H_0)=\alpha^3$. By Lemma \ref{lem:mean cubed density increment_small}, there is a subspace $H_1$ of bounded codimension with $b(H_1) \geq \alpha^3+(\alpha^3-\beta)/2$. If $\alpha^3 \geq 2^{8+8C_p}\beta$, by Lemma \ref{lem:mean cubed density increment_large}, we can instead find a subspace $H_1$ of bounded codimension with $b(H_1) \geq \beta(\alpha^3/\beta)^{1+1/(2C_p)} =(\alpha^3/\beta)^{1/(2C_p)}\alpha^3$.  
We then apply the arithmetic regularity lemma to find a subspace $H_2$ of $H_1$ of bounded codimension which is $\delta'$-regular, where $\delta'=\delta/8$. That is, all but at most a $\delta'$-fraction of the translates of $H_2$ are $\delta'$-regular, where an affine subspace $S$ is $\delta$-regular if the function $f$ is $\delta'$-close to the constant function $\mathbb{E}_{x \in S}[f(x)]$ on the subspace. By the counting lemma, Lemma \ref{lem:Counting lemma}, applied to each of the $\delta'$-regular translates of $H_2$, the average density of three-term arithmetic progressions with common difference in $H_2$ is at least $$\sum_{H_2+x~\textrm{is}~\delta'-\textrm{regular}} \mathbb{E}_{y \in H_2+x}[f(y)]^3-3\delta'\mathbb{E}_{y \in H_2+x}[f(y)] \geq b(H_2)-3\delta'-\delta'=b(H_2)-\delta/2,$$
where we used $\mathbb{E}_{y \in H_2+x}[f(y)] \leq 1$ and the density of translates of $H_2$ that are not $\delta'$-regular is at most $\delta'$. This includes the arithmetic progressions with common difference zero. As long as $H_2$ is sufficiently large, which holds as $n$ is sufficiently large to start with, then the $3$-AP density with common difference in $H_2$ is negligibly affected by whether the common difference zero arithmetic progressions are included or not. We thus get the average $3$-AP density with nonzero common difference in $H_2$ is at least $b(H_2)-\delta$, and hence there is a nonzero $d \in H_2$ for which the $3$-AP density with common difference $d$ is at least $b(H_2)-\delta$. By the lower bound $b(H_2) \geq b(H_1)$, which holds as $H_2$ is a subspace of $H_1$, this completes the proof. 

\vspace{0.2cm}
\noindent \textbf{Quasirandomness and arithmetic progressions}
\vspace{0.1cm} 

The study of quasirandomness in graphs and other structures have played an important role in combinatorics, number theory, and theoretical computer science. For graphs, Chung, Graham, Wilson \cite{CGW}, building on earlier work of Thomason \cite{Thomason1,Thomason2}, discovered many equivalent properties of graphs that are all shared (almost surely) by random graphs. Following this work, Chung and Graham \cite{ChGr} studied quasirandom properties in other combinatorial structures such as hypergraphs, permutations, boolean functions, and subsets of $\mathbb{Z}_N$. Here we focus on $\mathbb{F}_p^n$ with $p$ fixed as it is simpler due to the existence of subspaces, although it can be extended using Bohr sets to general abelian groups. 

For a set $A \subset \mathbb{F}_p^n$, we let $A:\mathbb{F}_p^n \rightarrow \{0,1\}$ denote the indicator function. That is $A(x) = 1$ if $x \in A$, and $A(x)=0$ otherwise. We let $N=p^n$. We say a set $A \subset \mathbb{F}_p^n$ is {\it $\epsilon$-quasirandom} if it satisfies the following property. 

\vspace{0.2cm}

\noindent $\mathcal{P}(\epsilon)$: For every affine subspace $H \subset \mathbb{F}_p^n$, we have $\left | |A \cap H| - |A||H|/N \right | \leq \epsilon N$.

\vspace{0.1cm}

As discovered by Chung and Graham \cite{ChGr} (in the setting of $\mathbb{Z}_N$), there are many equivalent properties up to changing $\epsilon$. For example, 
the set $A$ having all nonzero Fourier coefficients at most $\epsilon$ is such an equivalent property. Another example is that for all but at most $\epsilon N$ elements $x \in \mathbb{F}_p^n$, the size of the intersection of $A$ and its translate $A+x$ is within $\epsilon N$ of $|A|^2/N$. Yet another such property is that the Cayley sum graph of $A$ is $\epsilon$-quasirandom. This relates quasirandomness of subsets of $\mathbb{F}_p^n$ to the quasirandomness of an associated graph. 

A graph $H$ is forcing if, for every fixed $0<\alpha<1$, a graph with edge density $\alpha$ and $H$-density $\alpha^{e(H)}+o(1)$ is necessarily quasirandom. It is not hard to show that if $H$ is acyclic or not bipartite, then $H$ is not forcing, and the forcing conjecture is that all other graphs are forcing. However, Simonovits and S\'os \cite{SiSo} proved that if $H$ is a fixed graph with at least one edge, and $G$ is a graph on $n$ vertices 
such that every vertex subset $S$ has $\alpha^{e(H)}|S|^{v(H)}+o(n^{v(H)})$ ordered copies of $H$, then the graph is necessarily quasirandom.  In particular, if all linear-sized induced subgraphs of a graph 
have about the same density of triangles, then the graph is quasirandom. Their proof used Szemer\'edi's regularity lemma, and gave a weak (tower-type) bound on the dependency between the error parameter for this quasirandom property and the traditional quasirandom properties. They posed the problem of finding a new proof that avoids using Szemer\'edi's regularity lemma and gives a better dependency. This problem was recently solved by Conlon, Sudakov, and the first author \cite{CFS2}, giving a linear dependence for cliques, and a polynomial bound for other graphs.  

Proofs of Roth's theorem typically begin by observing that if the set is quasirandom, then a counting lemma shows that the total number of three-term arithmetic progressions is about the random bound. However, if the total number of three-term arithmetic progressions is about the random bound, then the set needs not be quasirandom. Motivating by this observation and the results of \cite{SiSo} and \cite{CFS2} in the case of triangles in graphs, it is natural to ask if there is a natural analogue of the quasirandom property of hereditary triangle counts in the arithmetic setting and how such a property relates to other notions of arithmetic quasirandomness quantitatively. In the arithmetic setting, we count three-term arithmetic progressions instead of triangles, and affine subspaces replace the role of vertex subsets. Indeed, we can prove that if the number of $3$-APs is not much larger than the random bound in any large affine subspace, then we do get quasirandomness. We have two different versions below. The first involves counting $3$-APs in a single affine subspace, and the second counts $3$-APs inside translates of a subspace, or equivalently, with common difference in the subspace. 

\vspace{0.2cm}

\noindent $\mathcal{Q}(\epsilon)$: For every affine subspace $H \subset \mathbb{F}_p^n$, we have $$\sum_{x,x+d \in H} A(x)A(x+d)A(x+2d) \leq |A|^3|H|^2/N^3 +\epsilon |H|N.$$


\noindent $\mathcal{R}(\epsilon)$: For every subspace $H \subset \mathbb{F}_p^n$, we have $$\sum_{x \in \mathbb{F}_p^n,d \in H} A(x)A(x+d)A(x+2d) \leq |A|^3|H|/N^2 + \epsilon N^2.$$


Both $\mathcal{Q}$ and $\mathcal{R}$ are quasirandom properties, that is, they are equivalent to $\mathcal{P}$ up to changing $\epsilon$. In particular, $\mathcal{P}(\epsilon)$ implies $\mathcal{Q}(8p^2\epsilon)$ and $\mathcal{R}(8p^2\epsilon)$. Conversely, $\mathcal{Q}(\delta)$ with $\delta=2^{-(p/\epsilon)^{O(1)}}$ implies $\mathcal{P}(\epsilon)$, whereas $\mathcal{R}(\delta)$ implies $\mathcal{P}(\epsilon)$ when $\delta^{-1}$ grows as a tower of $p$'s of height $\Theta(\log(1/\epsilon))$. Surprisingly, while the dependency between the quasirandom parameters of $\mathcal{Q}$ and $\mathcal{P}$ is only exponential and can be further improved, the dependency between the quasirandom parameters of $\mathcal{R}$ and $\mathcal{P}$ is of tower-type, which can be shown to be tight using the construction in the proof of Theorem \ref{thm:Lower bound for line density} in \cite{FPI}. This shows that, unlike for hereditary triangle counting where the dependency between the quasirandom parameters turns out to be linear, for property $\mathcal{R}$, the dependency turns out to be tower-type (of height logarithmic in $1/\epsilon$). We thus address with a somewhat unexpected answer the arithmetic analogue of the Simonovits-S\'os problem \cite{SiSo} (see also \cite{CoFoSu}) on the dependency between hereditary counts and other quasirandom properties. 

Further discussion and proofs are contained in \cite{FP2}.

\vspace{0.2cm}
\noindent \textbf{Popular restricted differences}
\vspace{0.1cm} 

There is now an extensive literature on strengthenings of Szemer\'edi's theorem in which the common difference lies in a particular set $S$. An early result of this sort is the Furstenberg-S\'ark\"ozy theorem \cite{Sarkozy}, which guarantees that any dense subset of the integers contains a pair of distinct elements whose difference is a perfect square. A result of Bergelson and Leibman \cite{BL96} implies that any dense subset of the integers contains a $k$-term arithmetic progression whose common difference is a perfect $r$th power. Another result of this type is when $S$ is the set of primes shifted by one, see \cite{FHK}. 
Quantitative bounds have also received much attention, see, e.g., Green \cite{Green02}. 

One naturally wonders whether similar strengthenings of Green's theorem hold, where the popular nonzero common difference must lie in a particular set $S$. A simple example in $\mathbb{F}_p^n$ is when $S$ is a subspace of codimension $D$, where $D$ is fixed. Indeed, such a result follows from the following somewhat stronger version of Theorem \ref{thm:Original theorem}.

\begin{thm}\label{subspacethm1} 
For each $\epsilon>0$, nonnegative integer $D$ and odd prime $p$, there is $\delta=\delta_p(\epsilon,D)>0$ such that for any subset of $\mathbb{F}_p^n$ and any subspace $H_0$ of codimension at most $D$, there is a subspace $H$ of $H_0$ with $|H| > \delta p^n$ such that the density of 
$3$-APs with common difference in $H$ is at least $b(H_0)-\epsilon$.
\end{thm}

The proof of the above theorem can be obtained directly from our quantitative improvement of Green's theorem, showing that we can take $\delta_p(\epsilon,D)$ so that $1/\delta_p(\epsilon,D)$ grows as a tower of $p$'s of height $O(\log(1/\epsilon))$ with a $D$ on top. Thus, there is a popular nonzero common difference in $S$ if $n$ is larger than a tower of $p$'s of height $O(\log(1/\epsilon))$ with a $D$ on top. We can also modify the lower bound construction in the proof of Theorem \ref{thm:Lower bound for line density} in \cite{FPI} by starting with the partition into translates of the subspace $S$, and get a lower bound showing that this is essentially tight. In particular, in $\mathbb{F}_p^n$ with $p$ fixed, if $S$ is a subspace of codimension $D$, to guarantee that for any set there is a nonzero $d \in S$ for which the density of $3$-APs with common difference $d$ is at least $\epsilon$ less than the random bound, the smallest dimension $n$ we need is a tower of $p$'s of height $\Theta(\log (1/\epsilon))$ with a $D$ on top. 

It would be interesting to prove other strengthenings of Theorem \ref{thm:Original theorem} with restricted differences. For example, what if $S$ is a random set with a given density? Even the threshold probability for Roth's theorem with a random restricted difference set is poorly understood (see, e.g., \cite{FLW}), so this is likely a challenging problem. 

\appendix 
\section*{Appendix}
\section{Estimates}

\subsection{\bf Bounds between constants in the proof of Theorem \ref{thm:epsilon_large}}
\label{appenbounds}
We collect here several useful bounds on the constants used in the proof of Theorem \ref{thm:epsilon_large}.  For $i \geq 3$, we have \begin{equation}\label{similar} u_{i}=\sum_{j=1}^i r_j = r_1+r_2+\sum_{j=0}^{\infty} (1+\gamma)^{-j} r_i - 
\sum_{j=1}^{\infty} (1+\gamma)^{-j} r_3 = r_1+ r_2 +(1+\gamma^{-1})r_i - \gamma^{-1}r_3 \leq (1+\gamma^{-1})r_i. \end{equation} 

We also have $r_2 = .7864m_1 \geq .7864 \cdot 9999 \log_p(1/\beta) \geq 7000 \log p$, so for $i \geq 3$ we have \begin{equation}\label{riineq} r_{i+1} = (1+\gamma)r_i=r_i + \gamma r_i \geq r_i+\gamma r_2 \geq r_i + 10000,\end{equation} 
and it follows by induction on $i$ that 
\begin{equation}\label{rilowerbound}
r_i \geq 10000i
\end{equation} 
for $i \geq 1$. 
Recall that $|A_{p,r}| \geq p^{.782r}$ for $p \geq 19$ and $r \geq 10^6$. Hence, $|A_{p,r_j}|/R_j \geq R_j^{-.218}$ and  
$\mu_j \geq \prod_{i=1}^j R_i^{-.218} = U_j^{-.218}$. Therefore,  \begin{equation}\label{sim0} m_i = N_{i-1} \left(|A_{p,r_{i}}|/R_{i}\right)^2 \mu_{i-1}^{3} \beta \geq N_{i-1}  R_i^{-.436} U_{i-1}^{-.654} \beta. \end{equation}
For $i=2$, this gives 
\begin{equation}\label{sim23} m_2 >  N_{1}R_2^{-.436}R_1^{-.654}p^{-(m_1+1)/10000} =  p^{m_1-.436 \cdot .7864 \cdot m_1-.654m_1-(m_1+1)/10000} > p^{m_1/400}.\end{equation}

We next prove by induction on $i$ that, for $i \geq 2$, we have \begin{equation}\label{sim1} m_i \geq 2 \cdot 10^6 n_{i-1} \geq 2 \cdot 10^6 m_{i-1}\end{equation} 
and 
\begin{equation}\label{sim2} m_i \geq 10^6 u_i \geq 10^6r_{i}.\end{equation}

Indeed, for $i=2$, we have $m_2 > p^{m_1/400} >2 \cdot 10^6m_1=2 \cdot 10^6n_1$, where we used that $m_1> 10^6$ and $p \geq 19$. 
We have $u_2 = r_1+r_2 < 2m_1$, so (\ref{sim1}) for $i=2$ implies (\ref{sim2}) for $i=2$. 

Next suppose that $i>2$ and we have proven (\ref{sim1}) and (\ref{sim2}) for smaller values. By (\ref{sim0}), $r_i \leq 10^5u_{i-1}$, and the induction hypothesis, we have $$m_i \geq p^{n_{i-1}-.436r_i-.654u_{i-1}}\beta \geq p^{n_{i-1}-10^5u_{i-1}}\beta \geq p^{.9n_{i-1}}\beta \geq p^{n_{i-1}/2} \geq 2 \cdot 10^6n_{i-1}.$$ We also have $$u_i =r_i+u_{i-1} \leq 10^5r_{i-1}+u_{i-1} \leq 10^6u_{i-1} \leq m_{i-1}  \leq m_i/10^6,$$ completing the induction proof of (\ref{sim1}) and (\ref{sim2}). 

By (\ref{sim0}) and (\ref{sim2}), for $i \geq 4$, we have \begin{equation}\label{sim4} m_i \geq p^{n_{i-1}-.436r_i-.654u_{i-1}}\beta  \geq p^{n_{i-1}- 10^6u_{i-2}}\beta \geq p^{m_{i-1}+m_{i-2}+m_1-10^6u_{i-2}}\beta \geq p^{m_{i-1}}.\end{equation} 

Recall that $|A_{p,r}| \geq ((p+1)/2.0001)^{r-2}$ for $p \geq 19$ and $r \geq 10^6$. As $R_i=p^{r_i}$, we therefore have $R_i/|A_{p,r_i}| \leq p^2(2.0001/(1+1/p))^{r_i}$ and so $1/\mu_i \leq p^{2i}(2.0001/(1+1/p))^{u_i}$.  
For $i \geq 2$, we have 
\begin{eqnarray} \label{sim6} \frac{m_iR_i}{\mu_{i-1}N_{i-1}} & = & \frac{N_{i-1} \left(|A_{p,r_{i}}|/R_{i}\right)^2 \mu_{i-1}^{3} \beta \cdot R_i}{\mu_{i-1}N_{i-1}} =  \mu_{i}^{2} \beta R_i \geq ep^5.
\end{eqnarray}
Indeed, for $i=2$, we have 
\begin{eqnarray*} \mu_i^2 \beta R_i & \geq & p^{-8} (2.0001/(1+1/p))^{-2u_2}\beta p^{r_2} \geq p^{-8} (2.0001/(1+1/p))^{-4.5433r_2}p^{-(r_1+1)/10000} p^{r_2} 
\\ & \geq & p^{-9}(p^{.9998}\left((1+1/p)/2.0001\right)^{4.5433})^{r_2} \geq p^{-9+r_2/200} \geq ep^5,\end{eqnarray*} where we used $\beta \ge p^{-(r_1+1)/10000}$ in the second inequality, and we used $r_1 \le 2r_2$ in the third inequality.  
For $i \geq 3$, we have 
\begin{eqnarray*}\mu_i^2 \beta R_i & \geq &  p^{-4i} 2.0001^{-2u_i}\beta p^{r_i} \geq p^{-4i} 2.0001^{-2(\log p)r_i/3}\beta p^{r_i} 
\geq p^{-4i}\beta p^{.3r_i}  \geq p^{-4i}p^{r_i/5} \\ & \geq & ep^5,\end{eqnarray*} 
where the last inequality is by (\ref{rilowerbound}).

\section*{Acknowledgements}
We would like to thank David Fox for an observation in the game of SET which led to the study of multidimensional cap sets and this paper. We are indebted to Yufei Zhao for many helpful comments, including observing that Schur's inequality simplifies the density increment argument. We would like to thank Ben Green for many helpful comments. We would also like to thank David Conlon, Joel Spencer, Ron Graham, and Terry Tao for helpful discussions related to the paper. 

\bibliographystyle{amsplain}


\begin{dajauthors}
\begin{authorinfo}[jf]
  Jacob Fox\\
  Department of Mathematics, \\
  Stanford University, Stanford, CA 94305, USA\\
  jacobfox\imageat{}stanford\imagedot{}edu \\
\end{authorinfo}
\begin{authorinfo}[htp]
  Huy Tuan Pham \\
  Department of Mathematics, \\
  Stanford University, Stanford, CA 94305, USA\\
  huypham\imageat{}stanford\imagedot{}edu \\
\end{authorinfo}
\end{dajauthors}
\end{document}